\documentclass{amsart}

\usepackage[T1]{fontenc}
\usepackage[latin1]{inputenc}
\usepackage{amsmath,amsbsy,amsthm,amstext,amsfonts,amssymb,amscd,amsxtra}
\usepackage{bm}
\usepackage{enumerate,mdwlist}
\usepackage{enumitem}
\usepackage{eucal}
\usepackage{mathrsfs}
\usepackage{stmaryrd}
\usepackage[all]{xy}
\usepackage{url}
\usepackage{xcolor}

\allowdisplaybreaks

\theoremstyle{plain}
\newtheorem{theorem}{Theorem}[section]
\newtheorem*{theoremA}{Main Theorem}

\newtheorem{lemma}[theorem]{Lemma}
\newtheorem{proposition}[theorem]{Proposition}
\newtheorem{proposition-definition}[theorem]{Proposition-Definition}
\newtheorem{corollary}[theorem]{Corollary}
\newtheorem{claim}[theorem]{Claim}

\theoremstyle{definition}
\newtheorem{definition}{Definition}[section]

\theoremstyle{remark}
\newtheorem{remark}[theorem]{Remark}

\newenvironment{Steps}{\setcounter{step}{0}}{}
\newcounter{step}
\newcommand{\Proofstep}{\par\refstepcounter{step}Step~\thestep.\space\ignorespaces}

\newcommand{\ZZ}{\mathbb{Z}}
\newcommand{\QQ}{\mathbb{Q}}
\newcommand{\RR}{\mathbb{R}}
\newcommand{\CC}{\mathbb{C}}
\newcommand{\FF}{\mathbb{F}}
\newcommand{\KK}{\mathbb{K}}

\newcommand{\OO}{\mathcal{O}}

\makeatletter
 
 \@addtoreset{equation}{section}
\makeatother

\makeatletter
 \newcommand{\esssup}{\mathop{\operator@font ess.sup}\displaylimits}
 \newcommand{\essinf}{\mathop{\operator@font ess.inf}\displaylimits}
\makeatother

\makeatletter
 \def\BIG#1{%
  {\hbox{$\left#1\vbox to20.5\p@{}\right.\n@space$}}}

\makeatother

\makeatletter 
 \let\@@pmod\pmod 
 \DeclareRobustCommand{\pmod}{\@ifstar\@pmods\@@pmod} 
 \def\@pmods#1{\mkern4mu({\operator@font mod}\mkern 6mu#1)} 
\makeatother

\newcommand{\Image}{\mathop{\mathrm{Image}}\nolimits}
\renewcommand{\leq}{\leqslant}
\renewcommand{\geq}{\geqslant}

\DeclareMathOperator{\Coker}{Coker}
\DeclareMathOperator{\aHz}{\widehat{\Gamma}^{\rm ss}}
\DeclareMathOperator{\aHzsm}{\widehat{\Gamma}^{\rm s}}
\DeclareMathOperator{\aHzf}{\widehat{\Gamma}^{\rm f}}
\DeclareMathOperator{\zdiv}{div}

\DeclareMathOperator{\Spec}{Spec}

\DeclareMathOperator{\Supp}{Supp}

\DeclareMathOperator{\Rat}{Rat}
\DeclareMathOperator{\DV}{\mathfrak{V}}
\DeclareMathOperator{\BC}{BC}
\DeclareMathOperator{\Pic}{Pic}
\DeclareMathOperator{\PicSch}{\mathbf{Pic}}
\DeclareMathOperator{\NS}{NS}
\DeclareMathOperator{\Div}{Div}

\DeclareMathOperator{\Cl}{Cl}

\DeclareMathOperator{\aDiv}{\widehat{Div}}
\DeclareMathOperator{\aCl}{\widehat{Cl}}
\DeclareMathOperator{\aBDiv}{\widehat{\mathbb{D}iv}}

\DeclareMathOperator{\ord}{ord}
\DeclareMathOperator{\vol}{vol}
\DeclareMathOperator{\avol}{\widehat{vol}}

\DeclareMathOperator{\rk}{rk}

\DeclareMathOperator{\Mod}{\widehat{Mod}}

\newcommand{\aC}{C_{\ell^1}}

\newcommand{\horiz}{{\rm horiz}}
\newcommand{\quot}{{\rm quot}}
\newcommand{\sub}{{\rm sub}}

\newcommand{\sbullet}{{\scriptscriptstyle\bullet}}

\newcommand{\essmin}{\qopname\relax m{ess.min}}
\newcommand{\ah}{\operatorname{\widehat{\ell}^{\ast}}}
\newcommand{\ahs}[1]{\widehat{\ell}_{#1}^{\ast}}
\newcommand{\als}{\operatorname{\widehat{\ell}^{s}}}
\newcommand{\alss}{\operatorname{\widehat{\ell}^{ss}}}
\newcommand{\aHzs}[1]{\operatorname{\widehat{\Gamma}^{#1}}}

\title[The continuity of volumes]{Adelic Cartier divisors with base conditions and the continuity of volumes}
\author{Hideaki Ikoma}
\address{Department of Mathematics, Kyoto University, Kyoto, 606-8502, Japan}
\email{ikoma@math.kyoto-u.ac.jp}
\begin{document}

\begin{abstract}
In the previous paper \cite{IkomaDiff1}, we introduced a notion of pairs of adelic $\RR$-Cartier divisors and $\RR$-base conditions.
The purpose of this paper is to propose an extended notion of adelic $\RR$-Cartier divisors that we call an $\ell^1$-adelic $\RR$-Cartier divisors, and to show that the arithmetic volume function defined on the space of pairs of $\ell^1$-adelic $\RR$-Cartier divisors and $\RR$-base conditions is continuous along the directions of $\ell^1$-adelic $\RR$-Cartier divisors.
\end{abstract}

\maketitle
\tableofcontents

\section{Introduction}

In Arakelov geometry, it is essentially important whether or not an adelic line bundle has a nonzero small section.
The asymptotic number of the small sections of high powers of an adelic line bundle $\overline{L}$ is encoded in an invariant which we call the \emph{arithmetic volume} of $\overline{L}$ and denote by $\avol(\overline{L})$.
The notion of arithmetic volume was first introduced by Moriwaki in a series of papers \cite{MoriwakiCont,MoriwakiContExt,MoriwakiAdelic}, where he proved that the arithmetic volume has many good properties such as the global continuity, the positive homogeneity, the birational invariance, etc.
A purpose of this paper is to give a generalization of Moriwaki's arithmetic volume function, and study its fundamental properties.

Let $K$ be a number field, and let $O_K$ be the ring of integers of $K$.
Let $M_K^{\rm fin}$ be the set of all the finite places of $K$.
For each $v\in M_K^{\rm fin}$, $K_v$ denotes the $v$-adic completion of $K$, and $\widetilde{K}_v$ denotes the residue field at $v$.
Let $X$ be a normal projective $K$-variety, and let $\Rat(X)$ be the field of rational functions on $X$.
For each $v\in M_K^{\rm fin}\cup\{\infty\}$, let $X_v^{\rm an}$ be the associated analytic space over $v$ (see section~\ref{subsubsec:Berkovich:analytic:space} for detail).
Let $D$ be an $\RR$-Cartier divisor on $X$ endowed with a $D$-Green function $g_{\infty}$ on $X_{\infty}^{\rm an}$.
To an $O_K$-model $(\mathscr{X},\mathscr{D})$ of $(X,D)$, we can associate an adelic $\RR$-Cartier divisor
\[
(\mathscr{D},g_{\infty})^{\rm ad}:=\left(D,\sum_{v\in M_K^{\rm fin}}g_v^{(\mathscr{X},\mathscr{D})}[v]+g_{\infty}[\infty]\right).
\]
We then define the \emph{$\ell^1$-distance} of two such models $(\mathscr{X}_1,\mathscr{D}_1)$ and $(\mathscr{X}_1,\mathscr{D}_2)$ as
\[
\sum_{v\in M_K^{\rm fin}}\sup_{x\in X_v^{\rm an}}\left|g_v^{(\mathscr{X}_1,\mathscr{D}_1)}(x)-g_v^{(\mathscr{X}_2,\mathscr{D}_2)}(x)\right|.
\]
For example, let $v_1,v_2,\dots$ be a sequence in $M_K^{\rm fin}$, and let $\mathscr{F}_i$ be the fiber of $\mathscr{X}$ over $v_i$.
The sequence of $O_K$-models
\[
\left(\left(\mathscr{X},\sum_{i=1}^n\frac{1}{2^i\log\sharp \widetilde{K}_{v_i}}\mathscr{F}_i\right)\right)_{n\geq 1}
\]
is then a Cauchy sequence in the $\ell^1$-distance.
However, it does not have a limit in the space of adelic $\RR$-Cartier divisors.
A basic principle of functional analysis tells us that function spaces should be complete, so we decide to extend the notion of adelic $\RR$-Cartier divisors so as the above sequence is to converge.
For each $v\in M_K^{\rm fin}\cup\{\infty\}$, we put $C(X_v^{\rm an})$ as the Banach algebra of $\RR$-valued continuous functions on $X_v^{\rm an}$ endowed with the supremum norm.
If $v=\infty$, we impose the condition that the functions in $C(X_{\infty}^{\rm an})$ are invariant under the complex conjugation map.
We define the space $\aC(X)$ of continuous functions on $X$ as the $\ell^1$-direct sum of the family $\left(C(X_v^{\rm an})\right)_{v\in M_K^{\rm fin}\cup\{\infty\}}$ endowed with the $\ell^1$-norm $\|\cdot\|_{\ell^1}$.
We say that a couple $\overline{D}=\left(D,\sum_{v\in M_K^{\rm fin}\cup\{\infty\}}g_v[v]\right)$ of an $\RR$-Cartier divisor $D$ and an adelic $D$-Green function $\sum_{v\in M_K^{\rm fin}\cup\{\infty\}}g_v[v]$ is an \emph{$\ell^1$-adelic $\RR$-Cartier divisor} if there exists an $O_K$-model $(\mathscr{X},\mathscr{D})$ of $(X,D)$ such that $\left\|\overline{D}-(\mathscr{D},g_{\infty})^{\rm ad}\right\|_{\ell^1}<+\infty$, and denote by $\aDiv_{\RR}^{\ell^1}(X)$ the $\RR$-vector space of all the $\ell^1$-adelic $\RR$-Cartier divisors on $X$.

There are several advantages of such an extension.
For example, the quotient space $\aCl_{\RR}^{\ell^1}(X)$ of $\aDiv_{\RR}^{\ell^1}(X)$ by the $\RR$-subspace generated by principal adelic Cartier divisors admits an essentially unique norm that makes $\aCl_{\RR}^{\ell^1}(X)$ into a Banach space (see section~\ref{subsec:definition:of:ell1:adelic}), which should be a proper arithmetic analogue of the space of numerical classes of $\RR$-Cartier divisors in algebraic geometry.
In particular, any surjective natural homomorphism $\aCl_{\RR}^{\ell^1}(X)\to\aCl_{\RR}^{\ell^1}(Y)$ is automatically an open mapping.
We expect that such a formalism will open a way for applying the powerful machinery of functional analysis, such as the duality theory, the semigroup theory, the spectral theory, etc., to the study of adelic $\RR$-Cartier divisors.

In the previous paper \cite{IkomaDiff1}, we introduced a notion of $\RR$-base conditions, and defined the arithmetic volumes for pairs of adelic $\RR$-Cartier divisors and $\RR$-base conditions.
An \emph{$\RR$-base condition} $\mathcal{V}$ on $X$ is defined as a formal $\RR$-linear combination
\[
 \mathcal{V}=\sum_{\nu}\nu(\mathcal{V})[\nu]
\]
such that $\nu$ are normalized discrete valuations of $\Rat(X)$ and such that $\nu(\mathcal{V})$ are zero for all but finitely many $\nu$.
A discrete valuation $\nu$ assigns to $D$ an \emph{order of vanishing along $\nu$} defined as $\nu(f)$, where $f$ is a local equation defining $D$ around the center $c_X(\nu)$ of $\nu$ on $X$.
We denote by $\BC_{\RR}(X)$ the $\RR$-vector space of all the $\RR$-base conditions on $X$.
Given a pair $(\overline{D};\mathcal{V})$ of $\overline{D}\in\aDiv_{\RR}^{\ell^1}(X)$ and $\mathcal{V}\in\BC_{\RR}(X)$, we can define
\[
 \als\left(\overline{D};\mathcal{V}\right):=\log\left(1+\sharp\left\{\phi\in\Rat(X)^{\times}\,:\,\text{$\overline{D}+\widehat{(\phi)}>0$, $\nu(D+(\phi))\geq\nu(\mathcal{V})$, $\forall\nu$}\right\}\right)
\]
as a nonnegative real number (see Proposition~\ref{prop:basic:ah:of:ell1:adelic:divisors}), and can define the \emph{arithmetic volume} of $(\overline{D};\mathcal{V})$ as
\[
 \avol\left(\overline{D};\mathcal{V}\right):=\lim_{\substack{m\in\ZZ, \\ m\to+\infty}}\frac{\als\left(m\overline{D};m\mathcal{V}\right)}{m^{\dim X+1}/(\dim X+1)!}
\]
(see Proposition~\ref{prop:finiteness:of:arithmetic:volumes}).
We will establish the following result (Theorem~\ref{thm:Continuity_Est3}).

\begin{theoremA}
Let $X$ be a normal, projective, and geometrically connected $K$-variety.
Let $V$ be a finite-dimensional $\RR$-subspace of $\aDiv_{\RR}^{\ell^1}(X)$, let $\|\cdot\|_V$ be a norm on $V$, let $\Sigma$ be a finite set of points on $X$, and let $B\in\RR_{>0}$.
Given any $\varepsilon>0$, there exists a $\delta>0$ such that
\[
 \left|\avol\left(\overline{D}+(0,\bm{f});\mathcal{V}\right)-\avol\left(\overline{E};\mathcal{V}\right)\right|\leq\varepsilon
\]
for every $\overline{D},\overline{E}\in V$ with $\max\left\{\left\|\overline{D}\right\|_V,\left\|\overline{E}\right\|_V\right\}\leq B$ and $\left\|\overline{D}-\overline{E}\right\|_V\leq\delta$, $\bm{f}\in \aC(X)$ with $\|\bm{f}\|_{\ell^1}\leq\delta$, and $\mathcal{V}\in\BC_{\RR}(X)$ with $\{c_X(\nu)\,:\,\nu(\mathcal{V})>0\}\subset\Sigma$.
\end{theoremA}

This paper comprises two parts.
First, in section~\ref{sec:prelim}, after showing preliminary results on base conditions (section~\ref{subsec:base:conditions}) and the change of norms (section~\ref{subsec:comparison:of:norms}), we prove in section~\ref{subsec:FundEst} the fundamental estimate of numbers of small sections of pairs, which is the key step to show Theorem~\ref{thm:Continuity_Est3}.

Next, section~\ref{sec:ell1adelic} will be devoted to introducing the notion of $\ell^1$-adelic $\RR$-Cartier divisors and showing Theorem~\ref{thm:Continuity_Est3}.
After recalling basic facts on the adelically normed vector spaces (section~\ref{subsubsec:prelim:adelically:normed}), the Berkovich analytic spaces (section~\ref{subsubsec:Berkovich:analytic:space}), and the $D$-Green functions (section~\ref{subsubsec:model:functions}), we will introduce basic definitions on the $\ell^1$-adelic setting in sections~\ref{subsec:space:of:continuous:functions} and \ref{subsec:definition:of:ell1:adelic}.
We will define the arithmetic volumes of pairs of $\ell^1$-adelic $\RR$-Cartier divisors and $\RR$-base conditions in section~\ref{subsec:arithmetic:volume:function} and give a proof of Theorem~\ref{thm:Continuity_Est3} in section~\ref{subsec:ContHomog}.

\subsection{Notation and terminology}\label{subsec:notation_and_terminology}

\subsubsection{}
\label{subsubsec:NT:R-modules}
Let $R$ be a ring and let $M$ be an $R$-module.
Given a subset $\Gamma$ of $M$, we denote by $\langle\Gamma\rangle_R$ the $R$-submodule of $M$ spanned by $\Gamma$.
In this paper, we adopt the dot-product notation, that is, for $\bm{a}=(a_1,\dots,a_r)\in R^r$ and $\bm{m}=(m_1,\dots,m_r)\in M^r$, we write
\[
 \bm{a}\cdot\bm{m}=a_1m_1+\dots+a_rm_r.
\]
Moreover, for $\bm{a}=(a_1,\dots,a_r)\in\RR^r$, $\|\bm{a}\|_1$ denotes the $\ell^1$-norm of $\bm{a}$:
\[
\|\bm{a}\|_1:=|a_1|+\dots+|a_r|.
\]
\medskip

\subsubsection{}
\label{subsubsec:NT:normed:Z-modules}
A \emph{normed $\ZZ$-module} $\overline{M}:=(M,\|\cdot\|)$ is a finitely generated $\ZZ$-module $M$ endowed with a norm $\|\cdot\|$ on $M_{\RR}=M\otimes_{\ZZ}\RR$.
For such an $\overline{M}$, we set
\[
\aHzsm(\overline{M}):=\left\{m\in M\,:\,\|m\otimes 1\|\leq 1\right\},\quad \als(\overline{M}):=\log\sharp\aHzsm(\overline{M})
\]
and
\[
\aHz(\overline{M}):=\left\{m\in M\,:\,\|m\otimes 1\|<1\right\},\quad \alss(\overline{M}):=\log\sharp\aHz(\overline{M}).
\]
Let $\ast=\text{s or ss}$.
The following properties are fundamental.
\begin{enumerate}
\item[(a)] Let $\overline{M}$ be a normed $\ZZ$-module and let
\[
0\to M'\to M\to M''\to 0
\]
be an exact sequence of $\ZZ$-modules.
We endow $M_{\RR}'$ (respectively, $M_{\RR}''$) with the subspace norm $\|\cdot\|_{\sub}$ (respectively, quotient norm $\|\cdot\|_{\quot}$) induced from $\overline{M}$.
One then has
\begin{equation}
\label{eqn:elementary:property:of:ah:exact:sequence}
\ah(\overline{M})\leq\ah(\overline{M}')+\ah(\overline{M}'')+3\rk M'+2\log(\rk M')!.
\end{equation}
In fact, if $\ast=\text{s}$, then the inequality is nothing but \cite[Proposition~2.1(4)]{MoriwakiCont} and, if $\ast=\text{ss}$, then it follows from the $\ast=\text{s}$ case by replacing $\|\cdot\|$ with $e^{\varepsilon}\|\cdot\|$ for $\varepsilon>0$ and taking $\varepsilon\downarrow 0$.

\item[(b)] If we replace $\|\cdot\|$ with $e^{-\lambda}\|\cdot\|$ for a $\lambda\in\RR_{\geq 0}$, then
\begin{equation}
\label{eqn:elementary:property:of:ah:rescaling}
\ah(M,\|\cdot\|) \leq \ah(M,e^{-\lambda}\|\cdot\|) \leq \ah(M,\|\cdot\|)+(\lambda+2)\rk M
\end{equation}
(see the proof of \cite[Lemma~2.9]{Yuan09}).

\item[(c)] If $M'$ is a $\ZZ$-submodule of $M$ with $M/M'$ torsion, then
\begin{equation}
\label{eqn:elementary:property:of:ah:rescaling:2}
\ah(M',\|\cdot\|) \leq \ah(M,\|\cdot\|) \leq \ah(M',\|\cdot\|)+\log\sharp(M/M')+2\rk M
\end{equation}
(see \cite[Lemma~1.3.3, (1.3.3.4)]{MoriwakiAdelic}).

\item[(d)] Let $M$ be a finitely generated $\ZZ$-module, and let $\|\cdot\|^{1},\|\cdot\|^{2}$ be two norms on $M_{\RR}$.
If $\|\cdot\|^{1}\leq\|\cdot\|^{2}$, then
\begin{equation}
\label{eqn:elementary:property:of:ah:norm:increase}
\ah(M,\|\cdot\|^{1})\geq \ah(M,\|\cdot\|^{2}).
\end{equation}

\item[(e)] Let $\langle\cdot,\cdot\rangle^1$ and $\langle\cdot,\cdot,\rangle^2$ be two Hermitian inner products on $M_{\CC}=M\otimes_{\ZZ}\CC$, and let $\|\cdot\|^1$ and $\|\cdot\|^2$ be the associated norms on $M_{\RR}$, respectively.
Let $e_1,\dots,e_l$ be any basis for $M_{\CC}$.
If $\|\cdot\|^1\leq\|\cdot\|^2$, then
\begin{align}
&\ah(M,\|\cdot\|^1)-\ah(M,\|\cdot\|^2)-3\rk M-2\log(\rk M)! \label{eqn:fundamental:property:of:norms3}\\
&\qquad\qquad\qquad\qquad\qquad\qquad\qquad\qquad \leq -\frac{1}{2}\log\frac{\det\left(\langle e_i,e_j\rangle^1\right)_{i,j}}{\det\left(\langle e_i,e_j\rangle^2\right)_{i,j}} \nonumber
\end{align}
(see \cite[Proposition~2.1(2)]{MoriwakiCont}).
The right-hand side does not depend on a specific choice of $e_1,\dots,e_l$.
The $\ast=\text{ss}$ case follows by the same arguments as in (a) above.
\end{enumerate}
We will also use the elementary inequalities
\[
\log n!\leq n\log n\quad\text{and}\quad \log n\leq n
\]
for every $n\in\ZZ_{>0}$.
\medskip

\subsubsection{}
\label{subsubsec:NT:non-Archimedean:fields}
Let $k$ be a field endowed with a non-Archimedean absolute value $|\cdot|$.
We write
\begin{equation}
k^{\circ}:=\{a\in k\,:\,|a|\leq 1\},\quad k^{\circ\circ}:=\{a\in k\,:\,|a|<1\},\quad\text{and}\quad \widetilde{k}:=k^{\circ}/k^{\circ\circ}.
\end{equation}
\medskip

\subsubsection{}
\label{subsubsec:NT:number:fields}
Let $K$ be a number field and let $O_K$ be the ring of integers of $K$.
Let $M_K^{\rm fin}$ be the set of all the finite places of $K$ and set
\begin{equation}
M_K:=M_K^{\rm fin}\cup\{\infty\}.
\end{equation}
Set $K_{\infty}:=\CC$ and set $|\alpha|_{\infty}:=\sqrt{\alpha\overline{\alpha}}$ for $\alpha\in\CC$.
For $v\in M_K^{\rm fin}$, we denote by $\mathfrak{p}_v$ the prime ideal of $O_K$ corresponding to $v$, by $K_v^{\circ}=\projlim_{n\in\ZZ_{>0}} O_K/\mathfrak{p}_v^n$ the $v$-adic completion of $O_K$, and by $K_v$ the quotient field of $K_v^{\circ}$.
We put
\begin{equation}
K_v^{\circ\circ}:=\mathfrak{p}_vK_v^{\circ}\quad\text{and}\quad\widetilde{K}_v:=K_v^{\circ}/K_v^{\circ\circ}.
\end{equation}
We will write a uniformizer of $K_v$ by $\varpi_v$.
We define the order of an $\alpha\in K_v^{\circ}$ as
\begin{equation}
\ord_v(\alpha):=\begin{cases}\max\left\{n\geq 0\,:\,\alpha\in (K_v^{\circ\circ})^n\right\} & \text{if $\alpha\neq 0$ and} \\ +\infty & \text{if $\alpha=0$,}\end{cases}
\end{equation}
and extend it to a map from $K_v$ by linearity.
The (normalized) $v$-adic absolute value on $K_v$ is defined as
\begin{equation}
|\alpha|_v:=\left(\sharp\widetilde{K}_v\right)^{-\ord_v(\alpha)}
\end{equation}
for $\alpha\in K_v$.

\section{Fundamental estimate}\label{sec:prelim}

\subsection{Base conditions}
\label{subsec:base:conditions}

\subsubsection{}
Let $F$ be a field.
A \emph{normalized discrete valuation} $\nu$ on $F$ is a surjective map from $F$ to $\ZZ\cup\{+\infty\}$ such that
\begin{enumerate}
\item[(a)] $\nu(f)=+\infty$ if and only if $f=0$,
\item[(b)] $\nu(f\cdot g)=\nu(f)+\nu(g)$ for $f,g\in F$, and
\item[(c)] $\nu(f+g)\geq\min\{\nu(f),\nu(g)\}$ for $f,g\in F$.
\end{enumerate}
We set $F_{\nu}^{\circ}:=\left\{f\in F\,:\,\nu(f)\geq 0\right\}$ and $F_{\nu}^{\circ\circ}:=\left\{f\in F\,:\,\nu(f)>0\right\}$.
Since $(F_{\nu}^{\circ})^{\times}=\left\{f\in F\,:\,\nu(f)=0\right\}$, $F_{\nu}^{\circ\circ}$ is a maximal ideal of $F_{\nu}^{\circ}$.
We denote by $\DV(F)$ the set of all the normalized discrete valuations on $F$.

\subsubsection{}
Let $S$ be a reduced, irreducible, and separated scheme and let $F:=\Rat(S)$ be the field of rational functions on $S$.
We assume the condition that,
\begin{enumerate}
\item[($\star$)] for every $\nu\in\DV(F)$, there exists a unique point $c_S(\nu)\in S$ such that
\[
 \OO_{S,c_S(\nu)}\subset F_{\nu}^{\circ}\quad\text{and}\quad\mathfrak{m}_{c_S(\nu)}=F_{\nu}^{\circ\circ}\cap \OO_{S,c_S(\nu)}.
\]
\end{enumerate}
We call $c_S(\nu)$ the \emph{center} of $\nu$ on $S$.
By the valuative criterion of properness, if $S$ is proper over $\Spec(\ZZ)$, then $S$ satisfies the condition ($\star$).

\begin{remark}
If $S$ is a proper variety over a field $k$, then we always assume that a valuation $\nu\in\DV(F)$ is trivial on $k$.
In particular, such a valuation always has a unique center $c_S(\nu)$ on $S$, and the condition ($\star$) is satisfied.
\end{remark}

An \emph{$\RR$-base condition} $\mathcal{V}$ on $S$ is defined as a finite formal sum
\[
 \mathcal{V}:=\sum_{\nu\in\DV(\Rat(S))}\nu(\mathcal{V})[\nu]
\]
with real coefficients $\nu(\mathcal{V})$.
We denote by $\BC_{\RR}(S)$ the $\RR$-vector space of all the $\RR$-base conditions on $S$.
We write $\mathcal{V}\geq 0$ if $\nu(\mathcal{V})\geq 0$ for every $\nu\in\DV(\Rat(S))$.

\subsubsection{}
Let $S$ be a reduced, irreducible, and projective scheme over a field or $\ZZ$.
Let $L$ be a line bundle on $S$, let $\nu\in\DV(\Rat(S))$, and let $\eta$ be a local frame of $L$ around $c_S(\nu)$.
Given any $s\in H^0(L)\setminus\{0\}$, one can write $s_{c_S(\nu)}=f\eta_{c_S(\nu)}$ with $f\in \OO_{S,c_S(\nu)}\setminus\{0\}$.
If $\eta'$ is another local frame of $L$ around $c_S(\nu)$, then $\eta'/\eta$ is invertible in $\OO_{S,c_S(\nu)}$.
So, if we write $s_{c_S(\nu)}=f'\eta_{c_S(\nu)}'$ with $f'\in\OO_{S,c_S(\nu)}\setminus\{0\}$, then $f/f'$ is invertible in $\OO_{S,c_S(\nu)}$ and $\nu(f)=\nu(f')$.
We define
\begin{equation}
 \nu(s):=\nu(f),
\end{equation}
which does not depend on a specific choice of $\eta$.
The following properties are obvious.
\begin{enumerate}
\item[(a)] If $s\in H^0(L)$ does not pass through $c_S(\nu)$, then $f$ is invertible around $c_S(\nu)$ and
\begin{equation}
\label{eqn:elementary:valuation:a}
 \nu(s)=0.
\end{equation}
\item[(b)] For $s,t\in H^0(L)$ and $\nu\in\DV(\Rat(S))$,
\begin{equation}
\label{eqn:elementary:valuation:b}
 \nu(s+t)\geq\min\left\{\nu(s),\nu(t)\right\}.
\end{equation}
\item[(c)] For two line bundles $L,M$ on $S$, $s\in H^0(L)$, $t\in H^0(M)$, and $\nu\in\DV(\Rat(S))$, one has
\begin{equation}
\label{eqn:elementary:valuation:c}
 \nu(s\otimes t)=\nu(s)+\nu(t).
\end{equation}
\end{enumerate}
For a pair $(L;\mathcal{V})$ of a line bundle $L$ and a $\mathcal{V}\in\BC_{\RR}(S)$, we set
\begin{equation}
 H^0(L;\mathcal{V}):=\left\{s\in H^0(L)\,:\,\text{$\nu(s)\geq\nu(\mathcal{V})$ for all $\nu\in\DV(\Rat(S))$}\right\}.
\end{equation}

\subsubsection{}
Let $k$ be a field or $\ZZ$.
Let $S$ be a reduced, irreducible, normal, and projective $k$-scheme, and let $\KK=\text{$\RR$, $\QQ$, or $\ZZ$}$.
A \emph{$\KK$-Cartier divisor} on $S$ is an $\KK$-linear combination
\[
D=a_1D_1+\dots+a_rD_r
\]
such that $a_i\in\KK$ and such that $D_i$ are Cartier divisors.
We denote by $\Div_{\KK}(S)$ the $\KK$-module of all the $\KK$-Cartier divisors on $S$.
If $\KK=\ZZ$, we simply write $\Div(S):=\Div_{\ZZ}(S)$ as usual.

Each $\nu\in\DV(\Rat(S))$ can extend to a map $\nu:\Rat(S)^{\times}\otimes_{\ZZ}\RR\to\RR$ by linearity.
Given a $D\in\Div_{\RR}(S)$ and a $\nu\in\DV(\Rat(S))$, we take a local equation $f$ defining $D$ around $c_S(\nu)$, and define
\begin{equation}
\nu(D):=\nu(f),
\end{equation}
which does not depend on a specific choice of $f$ (see \cite[Definition~2.2]{IkomaDiff1}).
Given a pair $(D;\mathcal{V})$ of an $\RR$-Cartier divisor $D$ and a $\mathcal{V}\in\BC_{\RR}(S)$, we set
\begin{equation}
H^0(D;\mathcal{V}):=\left\{\phi\in\Rat(S)^{\times}\,:\,\begin{array}{l}\text{$D+(\phi)\geq 0$ and $\nu(D+(\phi))\geq\nu(\mathcal{V})$} \\ \text{for all $\nu\in\DV(\Rat(S))$}\end{array}\right\}\cup\{0\},
\end{equation}
and define
\begin{equation}
\vol(D;\mathcal{V}):=\limsup_{\substack{m\in\ZZ, \\ m\to+\infty}}\frac{\rk_kH^0(mD;m\mathcal{V})}{m^{\dim S}/(\dim S)!}.
\end{equation}

\subsubsection{}
Let $\mathscr{X}$ be a projective arithmetic variety over $\Spec(\ZZ)$; namely, $\mathscr{X}$ is a reduced and irreducible scheme projective and flat over $\Spec(\ZZ)$.
Let $\mathscr{X}(\CC)$ be the complex analytic space associated to $\mathscr{X}_{\CC}:=\mathscr{X}\times_{\Spec(\ZZ)}\Spec(\CC)$.
A continuous Hermitian line bundle on $\mathscr{X}$ is a couple $(\mathscr{L},|\cdot|^{\overline{\mathscr{L}}})$ of a line bundle $\mathscr{L}$ on $\mathscr{X}$ and a continuous Hermitian metric $|\cdot|^{\overline{\mathscr{L}}}$ on $\mathscr{L}(\CC)$.

\begin{definition}
Let $(\overline{\mathscr{L}};\mathcal{V})$ be a pair of a continuous Hermitian line bundle $\overline{\mathscr{L}}$ on $\mathscr{X}$ and a $\mathcal{V}\in\BC_{\RR}(\mathscr{X})$.
The $\ZZ$-module
\[
 H^0(\mathscr{L};\mathcal{V})=\left\{s\in H^0(\mathscr{L})\,:\,\text{$\nu(s)\geq\nu(\mathcal{V})$ for all $\nu\in\DV(\Rat(\mathscr{X}))$}\right\}
\]
is endowed with the \emph{supremum norm} $\|\cdot\|_{\sup}^{\overline{\mathscr{L}}}$ defined as
\begin{equation}
\label{eqn:definition:of:supremum:norm}
\|s\|_{\sup}^{\overline{\mathscr{L}}}:=\sup_{x\in \mathscr{X}(\CC)}|s|^{\overline{\mathscr{L}}}(x)
\end{equation}
for $s\in H^0(\mathscr{L})$.
We will abbreviate
\begin{equation}
\label{eqn:ah:for:short:1}
 \ah\left(\overline{\mathscr{L}};\mathcal{V}\right):=\ah\left(H^0(\mathscr{L};\mathcal{V}),\|\cdot\|_{\sup}^{\overline{\mathscr{L}}}\right)
\end{equation}
for $\ast=\text{s and ss}$ (see Notation and terminology~\ref{subsubsec:NT:normed:Z-modules}).
\end{definition}

\subsection{Comparison of norms}
\label{subsec:comparison:of:norms}

Let $T$ be a finite disjoint union
\[
 T=\bigcup_{i=1}^lT_i
\]
of compact complex K\"ahler manifolds $T_i$ of pure dimension $d$.
Let $\omega$ be a K\"ahler form on $T$ and let $\Omega=\omega^{\wedge d}$ be the volume form on $T$ associated to $\omega$.
Let $\overline{M}=(M,h^{\overline{M}})$ be a line bundle $M$ on $T$ endowed with a $C^{\infty}$-Hermitian metric $h^{\overline{M}}$.
The \emph{supremum norm} of $s\in H^0(M)$ is defined as
\[
 \|s\|_{\sup}^{\overline{M}}:=\sup_{t\in T}|s|^{\overline{M}}(t),\quad\text{where}\quad |s|^{\overline{M}}(t):=\sqrt{h^{\overline{M}}(s,s)(t)}.
\]
The \emph{$L^2$-inner product} of $s_1,s_2\in H^0(M)$ with respect to $\Omega$ is defined as
\[
 \langle s_1,s_2\rangle_{L^2}^{\overline{M}}:=\int_Th^{\overline{M}}(s_1,s_2)(t)\,\Omega,
\]
and the \emph{$L^2$-norm} of $s$ is $\|s\|_{L^2}^{\overline{M}}:=\sqrt{\langle s,s\rangle_{L^2}^{\overline{M}}}$.
In the rest of this subsection, we study the effects of the change of norms to the numbers of small sections.

\subsubsection{}
The first one (Proposition~\ref{cor:Bernstein:Markov}) gives us a direct (not optimal) relation between the supremum norms and the subspace norms induced by a fixed nonzero section.

\begin{lemma}[see {\cite[Lemma~1.1.4]{MoriwakiCont}}]\label{lem:Bernstein:Markov:2}
Let $\overline{\bm{M}}=(\overline{M}_1,\dots,\overline{M}_r)$ be $C^{\infty}$-Hermitian line bundles on $T$ and let $U$ be an open subset of $T$.
Assume that $U\cap T_i$ are nonempty for all $i$.
There then exists a positive constant $C_1\geq 1$ depending only on $\overline{\bm{M}}$, $U$, and $T$ such that
\[
 \sup_{t\in U}|s|^{\bm{a}\cdot\overline{\bm{M}}}(t) \leq \|s\|_{\sup}^{\bm{a}\cdot\overline{\bm{M}}} \leq C_1^{\|\bm{a}\|_1}\sup_{t\in U}|s|^{\bm{a}\cdot\overline{\bm{M}}}(t)
\]
for every $\bm{a}\in\ZZ_{\geq 0}^r$ and $s\in H^0(\bm{a}\cdot\bm{M})$.
\end{lemma}

\begin{proposition}\label{cor:Bernstein:Markov}
Let $\overline{\bm{M}}=(\overline{M}_1,\dots,\overline{M}_r)$ and $\overline{E}$ be $C^{\infty}$-Hermitian line bundles on $T$.
Fix an $s_0\in H^0(E)\setminus\{0\}$.
The $\CC$-vector space $H^0(\bm{a}\cdot\bm{M})$ is endowed with the two norms $\|\cdot\|_{\sup}^{\bm{a}\cdot\overline{\bm{M}}}$ and $\|\cdot\|_{\sup,\sub(s_0^{\otimes b})}^{\bm{a}\cdot\overline{\bm{M}}+b\overline{E}}$, where $\|\cdot\|_{\sup,\sub(s_0^{\otimes b})}^{\bm{a}\cdot\overline{\bm{M}}+b\overline{E}}$ is the subspace norm induced from $\left(H^0(\bm{a}\cdot\bm{M}+bE),\|\cdot\|_{\sup}^{\bm{a}\cdot\overline{\bm{M}}+b\overline{E}}\right)$ via $H^0(\bm{a}\cdot\bm{M})\xrightarrow{\otimes s_0^{\otimes b}} H^0(\bm{a}\cdot\bm{M}+bE)$.
There then exists a constant $C_2\geq 1$ depending only on $\overline{\bm{M}}$, $(\overline{E},s_0)$, and $T$ such that
\[
\|\cdot\|_{\sup}^{\bm{a}\cdot\overline{\bm{M}}}\leq C_2^{\|\bm{a}\|_1+b}\|\cdot\|_{\sup,\sub(s_0^{\otimes b})}^{\bm{a}\cdot\overline{\bm{M}}+b\overline{E}}
\]
for every $\bm{a}\in\ZZ_{\geq 0}^r$ and $b\in\ZZ_{\geq 0}$.
\end{proposition}

\begin{proof}
We choose a nonempty open subset $U$ of $T$ such that
\[
 \delta:=\inf_{t\in U}|s_0|^{\overline{E}}(t)>0
\]
and such that $T_i\cap U\neq\emptyset$ for all $i$.
By Lemma~\ref{lem:Bernstein:Markov:2}, there is a $C_1\geq 1$ such that
\[
\|s\|_{\sup}^{\bm{a}\cdot\overline{\bm{M}}} \leq C_1^{\|\bm{a}\|_1} \sup_{t\in U}|s|^{\bm{a}\cdot\overline{\bm{M}}}(t)
\]
for every $\bm{a}\in\ZZ_{\geq 0}^r$ and $s\in H^0(\bm{a}\cdot\bm{M})$.
Hence
\begin{align*}
\|s\|_{\sup}^{\bm{a}\cdot\overline{\bm{M}}} &\leq \delta^{-b}C_1^{\|\bm{a}\|_1} \sup_{t\in U}|s\otimes s_0^{\otimes b}|^{\bm{a}\cdot\overline{\bm{M}}+b\overline{E}}(t) \\
&\leq \max\{\delta^{-1},C_1\}^{\|\bm{a}\|_1+b}\|s\otimes s_0^{\otimes b}\|_{\sup}^{\bm{a}\cdot\overline{\bm{M}}+b\overline{E}}
\end{align*}
for every $\bm{a}\in\ZZ_{\geq 0}^r$, $b\in\ZZ_{\geq 0}$, and $s\in H^0(\bm{a}\cdot\bm{M})$.
\end{proof}

\subsubsection{}
Let $T$ and $\Omega$ be as above, and consider the $L^2$-norms with respect to $\Omega$.
Let $\overline{M}=(M,|\cdot|^{\overline{M}})$ be a $C^{\infty}$-Hermitian line bundle on $T$, let $V$ be a linear series belonging to $M$, and let $e_{1},\dots, e_{l}$ be an $L^2$-orthonormal basis for $V$.
We define the \emph{Bergman distortion function} $\beta(V;\overline{M},\Omega)$ as
\begin{equation}
 \beta(V;\overline{M},\Omega)(x):=\sum_{i=1}^{l}|e_{i}|^{\overline{M}}(x)^2
\end{equation}
for $x\in T$.
It is easy to see that $\beta(V;\overline{M},\Omega)$ does not depend on a specific choice of $e_{1},\dots, e_{l}$.
If $V=H^0(M)$, then we abbreviate
\begin{equation}
\beta(\overline{M}):=\beta(H^0(M);\overline{M},\Omega)
\end{equation}
for simplicity.
The distortion function has the following elementary properties.
\begin{enumerate}
\renewcommand{\labelenumi}{(\alph{enumi})}
\item If $W$ is a linear series containing $V$, then $\beta(V;\overline{M},\Omega)\leq\beta(W;\overline{M},\Omega)$.
\item For a $c\in\RR_{>0}$, $\beta(V;\overline{M},c\Omega)=c^{-1}\beta(V;\overline{M},\Omega)$.
\suspend{enumerate}

\begin{proposition}[{\cite[Theorem~1.2.1]{MoriwakiContExt}}]
\label{prop:distortion:corollary}
Let $\overline{A}$ and $\overline{\bm{B}}=(\overline{B}_1,\dots,\overline{B}_r)$ be $C^{\infty}$-Hermitian line bundles on $T$ such that $A$ and $\bm{B}$ are all ample and such that the Hermitian metrics are all pointwise positive definite.
Suppose that the volume form is given as $\Omega:=c_1(\overline{A})^{\wedge d}$.
There then exists a constant $C_3>0$ such that
\[
\left\|\beta\left(a\overline{A}-\bm{b}\cdot\overline{\bm{B}}\right)\right\|_{\sup}\leq C_3a^{d}
\]
for every $a\in\ZZ_{>0}$ and $\bm{b}\in\ZZ^r_{\geq 0}$.
\end{proposition}

\subsubsection{}
Let $\mathscr{X}$ be a projective arithmetic variety over $\Spec(\ZZ)$, let $\overline{\mathscr{M}}$ and $\overline{\mathscr{A}}$ be continuous Hermitian line bundles on $\mathscr{X}$, and fix an $s_0\in H^0(\mathscr{A})\setminus\{0\}$.
The $\ZZ$-module $H^0(\mathscr{M};\mathcal{V})$ is endowed with the supremum norm $\|\cdot\|_{\sup}^{\overline{\mathscr{M}}}$ and the $L^2$-norm $\|\cdot\|_{L^2}^{\overline{\mathscr{M}}}$.
Let $\|\cdot\|_{\sup,\sub(s_0)}^{\overline{\mathscr{M}}+\overline{\mathscr{A}}}$ (respectively, $\|\cdot\|_{L^2,\sub(s_0)}^{\overline{\mathscr{M}}+\overline{\mathscr{A}}}$) be the subspace norm induced via $H^0(\mathscr{M};\mathcal{V})\xrightarrow{\otimes s_0}H^0(\mathscr{M}+\mathscr{A})$; namely,
\begin{equation}
\|s\|_{\sup,\sub(s_0)}^{\overline{\mathscr{M}}+\overline{\mathscr{A}}}:=\|s\otimes s_0\|_{\sup}^{\overline{\mathscr{M}}+\overline{\mathscr{A}}}\quad\text{(respectively, }\|s\|_{L^2,\sub(s_0)}^{\overline{\mathscr{M}}+\overline{\mathscr{A}}}:=\|s\otimes s_0\|_{L^2}^{\overline{\mathscr{M}}+\overline{\mathscr{A}}}\text{)}
\end{equation}
for $s\in H^0(\mathscr{M};\mathcal{V})$.
For $\ast=\text{s and ss}$, we write
\begin{equation}
\ahs{\sub(s_0)}\left(\overline{\mathscr{M}};\mathcal{V}\right):=\ah\left(H^0(\mathscr{M};\mathcal{V}),\|\cdot\|_{\sup,\sub(s_0)}^{\overline{\mathscr{M}}+\overline{\mathscr{A}}}\right)
\end{equation}
for short.
The next one plays a key role in showing the main estimate in section~\ref{subsec:FundEst}.

\begin{theorem}\label{thm:Continuity_Est1}
Let $\mathscr{X}$ be a projective arithmetic variety of dimension $d+1$ over $\Spec(\ZZ)$.
We assume that the generic fiber $\mathscr{X}_{\QQ}$ is smooth over $\Spec(\QQ)$.
Let $\overline{\pmb{\mathscr{L}}}=(\overline{\mathscr{L}}_1,\dots,\overline{\mathscr{L}}_r)$ and $\overline{\mathscr{A}}$ be $C^{\infty}$-Hermitian line bundles on $\mathscr{X}$, and fix an $s_0\in\aHzsm(\overline{\mathscr{A}})\setminus\{0\}$.
If the Hermitian metrics of $\overline{\mathscr{L}}_1+\overline{\mathscr{A}},\dots,\overline{\mathscr{L}}_r+\overline{\mathscr{A}}$, and $\overline{\mathscr{A}}$ are all pointwise positive definite, then there exists a constant $C_4>0$ depending only on $\overline{\pmb{\mathscr{L}}}$, $(\overline{\mathscr{A}},s_0)$, and $\mathscr{X}$ such that
\[
\ahs{\sub(s_0^{\otimes b})}\left(\bm{a}\cdot\overline{\pmb{\mathscr{L}}};\mathcal{V}\right) \leq \ah\left(\bm{a}\cdot\overline{\pmb{\mathscr{L}}};\mathcal{V}\right)+C_4\|\bm{a}\|_1^{d}(b+\log\|\bm{a}\|_1)
\]
for $\ast=\mathrm{s},\,\mathrm{ss}$, $\bm{a}\in\ZZ_{\geq 0}^r$ with $\|\bm{a}\|_1>0$, $b\in\ZZ_{\geq 0}$, and $\mathcal{V}\in\BC_{\RR}(\mathscr{X})$.
\end{theorem}

\begin{proof}
We set the volume form as
\[
\Omega:=c_1\left(\overline{\mathscr{L}}_1+\dots+\overline{\mathscr{L}}_r+r\overline{\mathscr{A}}\right)^{\wedge d},
\]
and consider the $L^2$-norms with respect to $\Omega$.
By Proposition~\ref{prop:distortion:corollary}, there exists a constant $D_1>0$ such that
\begin{align*}
&\left\|\beta\left(H^0\left(\bm{a}\cdot\pmb{\mathscr{L}};\mathcal{V}\right);\bm{a}\cdot\overline{\pmb{\mathscr{L}}},\Omega\right)\right\|_{\sup} \\
& \quad \leq \left\|\beta\left(\|\bm{a}\|_1(\overline{\mathscr{L}}_1+\dots+\overline{\mathscr{L}}_r+r\overline{\mathscr{A}})-\sum_{i=1}^r(\|\bm{a}\|_1-a_i)(\overline{\mathscr{L}}_i+\overline{\mathscr{A}})-\|\bm{a}\|_1\overline{\mathscr{A}}\right)\right\|_{\sup} \\
& \quad \leq D_1\|\bm{a}\|_1^{d}
\end{align*}
for every $\bm{a}\in\ZZ_{\geq 0}^r$ with $\|\bm{a}\|_1>0$ and $\mathcal{V}\in\BC_{\RR}(\mathscr{X})$.

We fix an $L^2$-orthonormal basis $e_1,\dots,e_l$ for $H^0(\bm{a}\cdot\pmb{\mathscr{L}};\mathcal{V})$ in which the Hermitian form,
\[
(s,t)\mapsto \left\langle s\otimes s_0^{\otimes b},t\otimes s_0^{\otimes b}\right\rangle_{L^2}^{\bm{a}\cdot\overline{\pmb{\mathscr{L}}}+b\overline{\mathscr{A}}},
\]
is diagonalized.
By \cite[Corollary~1.1.2]{MoriwakiCont}, we can change the supremum norms to the $L^2$-norms up to error term $O(\|\bm{a}\|_1^d\log\|\bm{a}\|_1)$ (see \cite[Proof of Lemma~1.3.3]{MoriwakiContExt}).
Since
\[
\|\cdot\|_{L^2,\sub(s_0^{\otimes b})}^{\bm{a}\cdot\overline{\pmb{\mathscr{L}}}+b\overline{\mathscr{A}}}\leq\|\cdot\|_{L^2}^{\bm{a}\cdot\overline{\pmb{\mathscr{L}}}}\cdot\left(\|s_0\|_{\sup}^{\overline{\mathscr{A}}}\right)^b\leq\|\cdot\|_{L^2}^{\bm{a}\cdot\overline{\pmb{\mathscr{L}}}},
\]
we can apply the inequality \eqref{eqn:fundamental:property:of:norms3} and find a constant $D_2>0$ such that
\begin{align}
&\ahs{\sub(s_0^{\otimes b})}\left(\bm{a}\cdot\overline{\pmb{\mathscr{L}}};\mathcal{V}\right)-\ah\left(\bm{a}\cdot\overline{\pmb{\mathscr{L}}};\mathcal{V}\right)-D_2\|\bm{a}\|_1^d\log\|\bm{a}\|_1 \label{eqn:distortion:corollary:1}\\
&\qquad\qquad \leq -\frac{1}{2}\log\left(\frac{\det\left(\langle e_i\otimes s_0^{\otimes b},e_j\otimes s_0^{\otimes b}\rangle_{L^2}^{\bm{a}\cdot\overline{\pmb{\mathscr{L}}}+b\overline{\mathscr{A}}}\right)_{i,j}}{\det\left(\langle e_i, e_j \rangle_{L^2}^{\bm{a}\cdot\overline{\pmb{\mathscr{L}}}}\right)_{i,j}}\right) \nonumber\\
&\qquad\qquad =-\frac{1}{2}\sum_{i=1}^l\log\int_{\mathscr{X}(\CC)}\left(|e_i|^{\bm{a}\cdot\overline{\pmb{\mathscr{L}}}}\right)^2\cdot \left(|s_0|^{\overline{\mathscr{A}}}\right)^{2b}\,\Omega \nonumber
\end{align}
for every $\bm{a}\in\ZZ_{\geq 0}$ with $\|\bm{a}\|_1>0$, $b\in\ZZ_{\geq 0}$, and $\mathcal{V}\in\BC_{\RR}(\mathscr{X})$.

Since $\int_{\mathscr{X}(\CC)}\left(|e_i|^{\bm{a}\cdot\overline{\pmb{\mathscr{L}}}}\right)^2\,\Omega=1$ for each $i$ and $s_0$ does not vanish identically on each connected component of $\mathscr{X}(\CC)$, one can apply Jensen's inequality \cite[Theorem~3.3]{RudinTrilogy2} to the right-hand side of \eqref{eqn:distortion:corollary:1} and obtain
\begin{align*}
&\ahs{\sub(s_0^{\otimes b})}\left(\bm{a}\cdot\overline{\pmb{\mathscr{L}}};\mathcal{V}\right)-\ah\left(\bm{a}\cdot\overline{\pmb{\mathscr{L}}};\mathcal{V}\right)-D_2\|\bm{a}\|_1^d\log\|\bm{a}\|_1 \\
&\qquad\qquad\qquad\qquad \leq \sum_{i=1}^l\int_{\mathscr{X}(\CC)}\left(|e_i|^{\bm{a}\cdot\overline{\pmb{\mathscr{L}}}}\right)^2\cdot(-b\log|s_0|)\,\Omega \\
&\qquad\qquad\qquad\qquad \leq\left(D_1\int_{\mathscr{X}(\CC)}-\log|s_0|\,\Omega\right)\cdot \|\bm{a}\|_1^{d}b.
\end{align*}
\end{proof}

\subsection{Main estimate: the case of models}\label{subsec:FundEst}

Let $\mathscr{X}$ be a projective arithmetic variety over $\Spec(\ZZ)$ of dimension $d+1$.
Given a family $\overline{\pmb{\bm{\mathscr{L}}}}:=(\overline{\mathscr{L}}_1,\dots,\overline{\mathscr{L}}_r)$ of continuous Hermitian line bundles on $\mathscr{X}$ and an $\bm{a}=(a_1,\dots,a_r)\in\ZZ^r$, we write
\[
\bm{a}\cdot\overline{\pmb{\bm{\mathscr{L}}}}:=a_1\overline{\mathscr{L}}_1+\dots+a_r\overline{\mathscr{L}}_r,\quad\bm{a}\cdot\pmb{\bm{\mathscr{L}}}:=a_1\mathscr{L}_1+\dots+a_r\mathscr{L}_r,
\]
and $\|\bm{a}\|_1:=|a_1|+\dots+|a_r|$ as in Notation and terminology~\ref{subsubsec:NT:R-modules}.
The purpose of this section is to show the following estimate.

\begin{theorem}\label{thm:Continuity_Est2}
Let $\mathscr{X}$ be a projective arithmetic variety of dimension $d+1$ over $\Spec(\ZZ)$.
Assume that the generic fiber $\mathscr{X}_{\QQ}$ is smooth over $\Spec(\QQ)$.
Let $\overline{\pmb{\mathscr{L}}}:=(\overline{\mathscr{L}}_1,\dots,\overline{\mathscr{L}}_r)$ be a family of $C^{\infty}$-Hermitian line bundles on $\mathscr{X}$ and let $\Sigma$ be a finite set of points on $\mathscr{X}$.
Let $\overline{\mathscr{A}}$ be any continuous Hermitian line bundle on $\mathscr{X}$.
There then exists a constant $C>0$ depending only on $\mathscr{X}$, $\overline{\pmb{\mathscr{L}}}$, $\overline{\mathscr{A}}$, and $\Sigma$ such that
\[
\ah\left(\bm{a}\cdot\overline{\pmb{\mathscr{L}}}+b\overline{\mathscr{A}};\mathcal{V}\right)-\ah\left(\bm{a}\cdot\overline{\pmb{\mathscr{L}}};\mathcal{V}\right) \leq C\left((\|\bm{a}\|_1+|b|)^{d}|b|+\|\bm{a}\|_1^d\log\|\bm{a}\|_1\right)
\]
for every $\bm{a}\in\ZZ^r$ with $\|\bm{a}\|_1>0$, $b\in\ZZ$, and $\mathcal{V}\in\BC_{\RR}(\mathscr{X})$ with
\[
\{c_{\mathscr{X}}(\nu)\,:\,\nu(\mathcal{V})>0\}\subset\Sigma.
\]
\end{theorem}

\begin{proof}
We divide the proof into five steps.
\medskip

\begin{Steps}
\Proofstep
We may assume $\mathcal{V}\geq 0$.
By considering $\pm\overline{\mathscr{L}}_1,\dots,\pm\overline{\mathscr{L}}_r$, and $\pm\overline{\mathscr{A}}$, one can observe that it suffices to show the theorem for $\bm{a}\in\ZZ_{\geq 0}^r$ with $\|\bm{a}\|_1>0$ and $b\in\ZZ_{\geq 0}$.
Moreover, if the theorem is true for an $\overline{\mathscr{A}}$, then it is also true for any $\overline{\mathscr{A}}'$ with $\overline{\mathscr{A}}'\leq\overline{\mathscr{A}}$ in place of $\overline{\mathscr{A}}$.
Hence we can assume without loss of generality that $\overline{\mathscr{A}}$ has the following four properties.
\begin{enumerate}
\item[(a)] $\mathscr{A}$ is ample on $\mathscr{X}$.
\item[(b)] The Hermitian metric of $\overline{\mathscr{A}}$ is $C^{\infty}$, and the Hermitian metrics of $\overline{\mathscr{L}}_1+\overline{\mathscr{A}},\dots,\overline{\mathscr{L}}_r+\overline{\mathscr{A}}$, and $\overline{\mathscr{A}}$ are all pointwise positive definite.
\item[(c)] For every $n\gg 1$, $\langle\aHz(n\overline{\mathscr{A}})\rangle_{\ZZ}=H^0(n\mathscr{A})$ (see Notation and terminology~\ref{subsubsec:NT:R-modules} and \cite[Lemma~5.3]{IkomaRem}).
\item[(d)] There is a nonzero small section $s_0\in\aHzsm(\overline{\mathscr{A}})$ such that $\zdiv(s_0)_{\QQ}$ is smooth and such that $s_0$ does not pass through any point in $\Sigma$.
\end{enumerate}
\medskip

\Proofstep
For each $k\in\ZZ_{>0}$, we set
\begin{equation}
k\mathscr{Y}:=\zdiv(s_0^{\otimes k}).
\end{equation}
For $\bm{a}\in\ZZ^r$ and $b\in\ZZ$, we consider the $\ZZ$-module
\begin{multline}
 H^0_{\mathscr{X}|k\mathscr{Y}}(\bm{a}\cdot\pmb{\bm{\mathscr{L}}}+b\mathscr{A};\mathcal{V}) \\
 :=\Image\left(H^0(\bm{a}\cdot\pmb{\bm{\mathscr{L}}}+b\mathscr{A};\mathcal{V})\to H^0((\bm{a}\cdot\pmb{\bm{\mathscr{L}}}+b\mathscr{A})|_{k\mathscr{Y}})\right)
\end{multline}
endowed with the quotient norm $\|\cdot\|_{\sup,\quot(\mathscr{X}|k\mathscr{Y})}^{\bm{a}\cdot\overline{\pmb{\mathscr{L}}}+b\overline{\mathscr{A}}}$ induced from
\[
\left(H^0(\bm{a}\cdot\pmb{\mathscr{L}}+b\mathscr{A};\mathcal{V}),\|\cdot\|_{\sup}^{\bm{a}\cdot\overline{\pmb{\mathscr{L}}}+b\overline{\mathscr{A}}}\right).
\]

By abuse of notation, we will abbreviate, for $\sbullet=\sub(\text{---})$, $\quot(\mathscr{X}|\text{---})$, etc.,
\[
\ahs{\sbullet}\left(\bm{a}\cdot\overline{\pmb{\mathscr{L}}}+b\overline{\mathscr{A}};\mathcal{V}\right):=\ah\left(H^0_{?}(\bm{a}\cdot\pmb{\mathscr{L}}+b\mathscr{A};\mathcal{V}),\|\cdot\|_{\sup,\sbullet}^{?}\right)
\]
for simplicity, which in practice will cause no confusion.

 By Snapper's theorem \cite[page 295]{Kleiman}, one can find a constant $C>0$ depending only on $\pmb{\mathscr{L}}$, $\mathscr{A}$, $\mathscr{X}$, and $\mathscr{Y}$ such that
\begin{equation}
\label{eqn:Snappers:theorem}
h^0(\bm{a}\cdot\pmb{\mathscr{L}}+b\mathscr{A})\leq C(\|\bm{a}\|_1+b)^d\quad\text{and}\quad h^0((\bm{a}\cdot\pmb{\mathscr{L}}+b\mathscr{A})|_{\mathscr{Y}})\leq C(\|\bm{a}\|_1+b)^{d-1}
\end{equation}
for every $\bm{a}\in\ZZ_{\geq 0}^r$ and $b\in\ZZ_{\geq 0}$ with $\|\bm{a}\|_1+b>0$.

In the rest of the proof, the constant $C$ will be fittingly changed without explicit mentioning of it.

\begin{claim}
\label{clm:Snapper:and:aSnapper}
There exists a constant $C>0$ depending only on $\overline{\pmb{\mathscr{L}}}$, $\overline{\mathscr{A}}$, and $\mathscr{Y}$ such that
\[
\ahs{\quot(\mathscr{X}|\mathscr{Y})}\left(\bm{a}\cdot\overline{\pmb{\mathscr{L}}}+b\overline{\mathscr{A}};\mathcal{V}\right)\leq C(\|\bm{a}\|_1+b)^{d}
\]
for every $\bm{a}\in\ZZ_{\geq 0}^r$ and $b\in\ZZ_{\geq 0}$ with $\|\bm{a}\|_1+b>0$ and $\mathcal{V}\in\BC_{\RR}(\mathscr{X})$.
\end{claim}

\begin{proof}
It suffices to show the estimate
\[
\ah\left((\bm{a}\cdot\overline{\pmb{\mathscr{L}}}+b\overline{\mathscr{A}})|_{\mathscr{Y}}\right)\leq C(\|\bm{a}\|_1+b)^{d}
\]
for $\bm{a}\in\ZZ_{\geq 0}^r$ and $b\in\ZZ_{\geq 0}$ with $\|\bm{a}\|_1+b>0$.
Let $\mathscr{Y}_{\horiz}$ be the horizontal part of $\mathscr{Y}$, that is, the Zariski closure of $\mathscr{Y}_{\QQ}$ in $\mathscr{X}$.
Let $\mathcal{I}$ (respectively, $\mathcal{I}_{\horiz}$) be the ideal sheaf defining $\mathscr{Y}$ (respectively, $\mathscr{Y}_{\horiz}$) in $\mathscr{X}$.
By the properties (a) and (c) of Step 1, one finds an $n\in\ZZ_{>0}$ and nonzero small sections $t_i\in\aHzsm(n\overline{\mathscr{A}}-\overline{\mathscr{L}}_i)$ for $i=1,\dots,r$ such that each $t_i$ does not pass through any associated point of $\OO_{\mathscr{X}}/\mathcal{I}_{\horiz}$ and $\mathcal{I}_{\horiz}/\mathcal{I}$.

First, one finds a constant $C>0$ such that
\begin{align}
\ah\left((\bm{a}\cdot\overline{\pmb{\mathscr{L}}}+b\overline{\mathscr{A}})|_{\mathscr{Y}_{\horiz}}\right) & \leq \ah\left((n\|\bm{a}\|_1+b)\overline{\mathscr{A}}|_{\mathscr{Y}_{\horiz}}\right) \label{eqn:Snapper:and:aSnapper1}\\
& \leq C(\|\bm{a}\|_1+b)^d \nonumber
\end{align}
for every $\bm{a}\in\ZZ_{\geq 0}^r$ and $b\in\ZZ_{\geq 0}$ with $\|\bm{a}\|_1+b>0$ (see for example \cite[Theorem~2.8]{Boucksom_Chen}).

Next, $\mathcal{I}_{\horiz}/\mathcal{I}$ is a torsion sheaf having support of dimension $\leq d$.
So, by Snapper's theorem, one has
\begin{align}
\log\sharp H^0\left((\bm{a}\cdot\pmb{\mathscr{L}}+b\mathscr{A})\otimes\mathcal{I}_{\horiz}/\mathcal{I}\right) & \leq \log\sharp H^0\left((n\|\bm{a}\|_1+b)\mathscr{A}\otimes\mathcal{I}_{\horiz}/\mathcal{I}\right) \label{eqn:Snapper:and:aSnapper2}\\
& \leq C(\|\bm{a}\|_1+b)^d \nonumber
\end{align}
for every $\bm{a}\in\ZZ_{\geq 0}^r$ and $b\in\ZZ_{\geq 0}$ with $\|\bm{a}\|_1+b>0$.

Applying \eqref{eqn:elementary:property:of:ah:exact:sequence} to the exact sequence
\[
0\to(\bm{a}\cdot\pmb{\mathscr{L}}+b\mathscr{A})\otimes(\mathcal{I}_{\horiz}/\mathcal{I})\to(\bm{a}\cdot\pmb{\mathscr{L}}+b\mathscr{A})|_{\mathscr{Y}}\to(\bm{a}\cdot\pmb{\mathscr{L}}+b\mathscr{A})|_{\mathscr{Y}_{\horiz}}\to 0,
\]
one obtains, by \eqref{eqn:Snapper:and:aSnapper1} and \eqref{eqn:Snapper:and:aSnapper2},
\begin{align*}
&\ah\left((\bm{a}\cdot\overline{\pmb{\mathscr{L}}}+b\overline{\mathscr{A}})|_{\mathscr{Y}}\right)\\
&\qquad\quad \leq\ah\left((\bm{a}\cdot\overline{\pmb{\mathscr{L}}}+b\overline{\mathscr{A}})|_{\mathscr{Y}_{\horiz}}\right)+\log\sharp H^0\left((\bm{a}\cdot\pmb{\mathscr{L}}+b\mathscr{A})\otimes(\mathcal{I}_{\horiz}/\mathcal{I})\right) \\
&\qquad\quad \leq C(\|\bm{a}\|_1+b)^d
\end{align*}
for every $\bm{a}\in\ZZ_{\geq 0}^r$ and $b\in\ZZ_{\geq 0}$ with $\|\bm{a}\|_1+b>0$.
\end{proof}
\medskip

\Proofstep
We begin the estimation with the following claim.

\begin{claim}
\label{clm:main:exact:sequence}
Let $k\in\ZZ_{>0}$, let $\mathscr{M}$ be a line bundle on $\mathscr{X}$, and let $\mathcal{V}\in\BC_{\RR}(\mathscr{X})$ such that $\{c_{\mathscr{X}}(\nu)\,:\,\nu(\mathcal{V})>0\}\subset\Sigma$.
\begin{enumerate}
\item Tensoring by $s_0^{\otimes k}$ induces a homomorphism $H^0(\mathscr{M};\mathcal{V})\to H^0(\mathscr{M}+k\mathscr{A};\mathcal{V})$.
\item The sequence
\[
 0\to H^0(\mathscr{M};\mathcal{V})\xrightarrow{\otimes s_0^{\otimes k}} H^0(\mathscr{M}+k\mathscr{A};\mathcal{V})\xrightarrow{q} H^0_{\mathscr{X}|k\mathscr{Y}}(\mathscr{M}+k\mathscr{A};\mathcal{V})\to 0
\]
is exact.
\end{enumerate}
\end{claim}

\begin{proof}
(1): Let $t\in H^0(\mathscr{M};\mathcal{V})$.
By the property (d) of Step 1, one has $\nu(s_0)=0$ for every $\nu$ with $\nu(\mathcal{V})>0$.
By \eqref{eqn:elementary:valuation:c},
\[
 \nu(t\otimes s_0^{\otimes k})=\nu(t)+k\nu(s_0)\begin{cases} =\nu(t)\geq\nu(\mathcal{V}) & \text{if $\nu(\mathcal{V})>0$ and} \\ \geq 0 & \text{if $\nu(\mathcal{V})=0$} \end{cases}
\]
for every $\nu\in\DV(\Rat(\mathscr{X}))$; hence $t\otimes s_0^{\otimes k}\in H^0(\mathscr{M}+k\mathscr{A};\mathcal{V})$.

(2): Suppose that $t\in H^0(\mathscr{M}+k\mathscr{A};\mathcal{V})$ satisfies $q(t)=0$; hence one finds a $t_0\in H^0(\mathscr{M})$ such that $t=t_0\otimes s_0^{\otimes k}$.
By \eqref{eqn:elementary:valuation:c} and the property (d) of Step 1,
\[
 \nu(t_0)\begin{cases} =\nu(t)\geq\nu(\mathcal{V}) & \text{if $\nu(\mathcal{V})>0$ and} \\ \geq 0 & \text{if $\nu(\mathcal{V})=0$} \end{cases}
\]
for every $\nu\in\DV(\Rat(\mathscr{X}))$; hence $t_0\in H^0(\mathscr{M};\mathcal{V})$.
\end{proof}

If $b=0$, then the theorem is obvious, so that we can assume $b>0$.
We apply \eqref{eqn:elementary:property:of:ah:exact:sequence} to the exact sequence
\begin{equation}
 0\to H^0(\bm{a}\cdot\pmb{\bm{\mathscr{L}}};\mathcal{V})
 \xrightarrow{\otimes s_0^{\otimes b}} H^0(\bm{a}\cdot\pmb{\bm{\mathscr{L}}}+b\mathscr{A};\mathcal{V})
 \longrightarrow H^0_{\mathscr{X}|b\mathscr{Y}}(\bm{a}\cdot\pmb{\bm{\mathscr{L}}}+b\mathscr{A};\mathcal{V})\to 0,
\end{equation}
and obtain, by \eqref{eqn:Snappers:theorem} and Theorem~\ref{thm:Continuity_Est1},
\begin{align}
& \ah\left(\bm{a}\cdot\overline{\pmb{\mathscr{L}}}+b\overline{\mathscr{A}};\mathcal{V}\right) \label{eqn:fund_est_1}\\
& \quad \leq\ahs{\sub(s_0^{\otimes b})}\left(\bm{a}\cdot\overline{\pmb{\mathscr{L}}};\mathcal{V}\right)+\ahs{\quot(\mathscr{X}|b\mathscr{Y})}\left(\bm{a}\cdot\overline{\pmb{\mathscr{L}}}+b\overline{\mathscr{A}};\mathcal{V}\right) \nonumber\\
& \qquad\qquad\qquad\qquad\qquad\qquad\qquad\qquad\qquad\qquad +C\|\bm{a}\|_1^{d}(1+\log\|\bm{a}\|_1) \nonumber\\
& \quad \leq\ah\left(\bm{a}\cdot\overline{\pmb{\mathscr{L}}};\mathcal{V}\right)+\ahs{\quot(\mathscr{X}|b\mathscr{Y})}\left(\bm{a}\cdot\overline{\pmb{\mathscr{L}}}+b\overline{\mathscr{A}};\mathcal{V}\right)+C\|\bm{a}\|_1^{d}(b+\log\|\bm{a}\|_1) \nonumber
\end{align}
for every $\bm{a}\in\ZZ_{\geq 0}^r$ with $\|\bm{a}\|_1>0$ and $b\in\ZZ_{>0}$.
\medskip

\Proofstep
We are going to estimate the middle term $\ahs{\quot(\mathscr{X}|b\mathscr{Y})}\left(\bm{a}\cdot\overline{\pmb{\mathscr{L}}}+b\overline{\mathscr{A}};\mathcal{V}\right)$ in the right-hand side of \eqref{eqn:fund_est_1}.
For each $k\in\ZZ_{>0}$, we identify $-k\mathscr{A}$ with an ideal sheaf of $\OO_{\mathscr{X}}$ via the morphism $-k\mathscr{A}\xrightarrow{\otimes s_0^{\otimes k}}\OO_{\mathscr{X}}$.
The inclusions $-(k+1)\mathscr{A}\xrightarrow{\otimes s_0} -k\mathscr{A}\xrightarrow{\otimes s_0^{\otimes k}}\OO_{\mathscr{X}}$ induce an injective morphism
\begin{multline*}
 \sigma_k:-k\mathscr{A}|_{\mathscr{Y}}=\Coker\left(-(k+1)\mathscr{A}\xrightarrow{\otimes s_0}-k\mathscr{A}\right) \\
 \to \Coker\left(-(k+1)\mathscr{A}\xrightarrow{\otimes s_0^{\otimes(k+1)}}\OO_{\mathscr{X}}\right)=\OO_{(k+1)\mathscr{Y}}.
\end{multline*}

\begin{claim}
For each $k\in\ZZ_{>0}$, $\sigma_k$ induces a homomorphism
\[
H^0_{\mathscr{X}|\mathscr{Y}}(\bm{a}\cdot\pmb{\mathscr{L}}+(b-k)\mathscr{A};\mathcal{V})\to H^0_{\mathscr{X}|(k+1)\mathscr{Y}}(\bm{a}\cdot\pmb{\mathscr{L}}+b\mathscr{A};\mathcal{V}).
\]
\end{claim}

\begin{proof}
This is obvious because $\sigma_k$ induces a homomorphism
\[
H^0\left((\bm{a}\cdot\pmb{\mathscr{L}}+(b-k)\mathscr{A})|_{\mathscr{Y}}\right)\to H^0\left((\bm{a}\cdot\pmb{\mathscr{L}}+b\mathscr{A})|_{(k+1)\mathscr{Y}}\right)
\]
and the diagram
\[
\xymatrix{
 H^0\left((\bm{a}\cdot\pmb{\mathscr{L}}+(b-k)\mathscr{A})|_{\mathscr{Y}}\right) \ar[r]^-{\sigma_k} & H^0\left((\bm{a}\cdot\pmb{\mathscr{L}}+b\mathscr{A})|_{(k+1)\mathscr{Y}}\right) \\
 H^0(\bm{a}\cdot\pmb{\mathscr{L}}+(b-k)\mathscr{A};\mathcal{V}) \ar[u] \ar[r]^-{\otimes s_0^{\otimes k}} & H^0(\bm{a}\cdot\pmb{\mathscr{L}}+b\mathscr{A};\mathcal{V}) \ar[u]
}
\]%
is commutative.
\end{proof}

A commutative diagram of $\OO_{\mathscr{X}}$-modules:
\[
\xymatrix{
 0 \ar[r] & -k\mathscr{A}|_{\mathscr{Y}} \ar[r]^-{\sigma_k} & \OO_{(k+1)\mathscr{Y}} \ar[r] & \OO_{k\mathscr{Y}} \ar[r] & 0 \\
 0 \ar[r] & -k\mathscr{A} \ar[u] \ar[r]^-{\otimes s_0^{\otimes k}} & \OO_{\mathscr{X}} \ar[u] \ar[r] & \OO_{k\mathscr{Y}} \ar@{=}[u] \ar[r] & 0,
}
\]
yields a commutative diagram of $\ZZ$-modules:
{\tiny
\[
\xymatrix{
 0 \ar[r] & H^0_{\mathscr{X}|\mathscr{Y}}(\bm{a}\cdot\pmb{\mathscr{L}}+(b-k)\mathscr{A};\mathcal{V}) \ar[r]^-{\sigma_k} & H^0_{\mathscr{X}|(k+1)\mathscr{Y}}(\bm{a}\cdot\pmb{\mathscr{L}}+b\mathscr{A};\mathcal{V}) \ar[r] & H^0_{\mathscr{X}|k\mathscr{Y}}(\bm{a}\cdot\pmb{\mathscr{L}}+b\mathscr{A};\mathcal{V}) \ar[r] & 0 \\
 0 \ar[r] & H^0(\bm{a}\cdot\pmb{\mathscr{L}}+(b-k)\mathscr{A};\mathcal{V}) \ar[r]^-{\otimes s_0^{\otimes k}} \ar[u] & H^0(\bm{a}\cdot\pmb{\mathscr{L}}+b\mathscr{A};\mathcal{V}) \ar[r] \ar[u] & H^0_{\mathscr{X}|k\mathscr{Y}}(\bm{a}\cdot\pmb{\mathscr{L}}+b\mathscr{A};\mathcal{V}) \ar[r] \ar@{=}[u] & 0.
}
\]}%
Since the right vertical arrow is an identity and the lower horizontal sequence is exact (see Claim~\ref{clm:main:exact:sequence}),
one sees that the upper horizontal sequence of the diagram is also exact.
Applying \eqref{eqn:elementary:property:of:ah:exact:sequence} to the upper horizontal sequence, one obtains
\begin{align}
&\ahs{\quot(\mathscr{X}|(k+1)\mathscr{Y})}\left(\bm{a}\cdot\overline{\pmb{\mathscr{L}}}+b\overline{\mathscr{A}};\mathcal{V}\right) \label{eqn:fund_est_2}\\
 & \qquad \leq\ahs{\quot(\mathscr{X}|(k+1)\mathscr{Y}),\sub(\sigma_k)}\left(\bm{a}\cdot\overline{\pmb{\mathscr{L}}}+(b-k)\overline{\mathscr{A}};\mathcal{V}\right) \nonumber\\
 &\qquad +\ahs{\quot(\mathscr{X}|k\mathscr{Y})}\left(\bm{a}\cdot\overline{\pmb{\mathscr{L}}}+b\overline{\mathscr{A}};\mathcal{V}\right)+C(\|\bm{a}\|_1+b)^{d-1}\left(1+\log(\|\bm{a}\|_1+b)\right) \nonumber
\end{align}
for every $\bm{a}\in\ZZ_{\geq 0}^r$ and $b,k\in \ZZ_{>0}$ with $\|\bm{a}\|_1+b>0$ and $k\leq b$ (see \eqref{eqn:Snappers:theorem}).
\medskip

\Proofstep
By applying \cite[Lemma~3.4(2)]{MoriwakiCont} to the right square of the commutative diagram
{\tiny
\[
\xymatrix{
0 \ar[r] & H^0(\bm{a}\cdot\pmb{\mathscr{L}}+(b-k-1)\mathscr{A};\mathcal{V}) \ar[r]^-{\otimes s_0^{\otimes (k+1)}} & H^0(\bm{a}\cdot\pmb{\mathscr{L}}+b\mathscr{A};\mathcal{V}) \ar[r] & H^0_{\mathscr{X}|(k+1)\mathscr{Y}}(\bm{a}\cdot\pmb{\mathscr{L}}+b\mathscr{A};\mathcal{V}) \ar[r] & 0 \\
0 \ar[r] & H^0(\bm{a}\cdot\pmb{\mathscr{L}}+(b-k-1)\mathscr{A};\mathcal{V}) \ar[r]^-{\otimes s_0} \ar@{=}[u] & H^0(\bm{a}\cdot\pmb{\mathscr{L}}+(b-k)\mathscr{A};\mathcal{V}) \ar[r] \ar[u]_-{\otimes s_0^{\otimes k}} & H^0_{\mathscr{X}|\mathscr{Y}}(\bm{a}\cdot\pmb{\mathscr{L}}+(b-k)\mathscr{A};\mathcal{V}) \ar[r] \ar[u]_-{\sigma_k} & 0,
}
\]}%
and by using Proposition~\ref{cor:Bernstein:Markov}, one can get a constant $D$ with $0<D\leq 1$ such that
\begin{align*}
 \|\cdot\|_{\sup,\quot(\mathscr{X}|(k+1)\mathscr{Y}),\sub(\sigma_k)}^{\bm{a}\cdot\overline{\pmb{\mathscr{L}}}+b\overline{\mathscr{A}}} & =\|\cdot\|_{\sup,\sub(s_0^{\otimes k}),\quot(\mathscr{X}|\mathscr{Y})}^{\bm{a}\cdot\overline{\pmb{\mathscr{L}}}+b\overline{\mathscr{A}}} \\
 & \geq D^{\|\bm{a}\|_1+b}\|\cdot\|_{\sup,\quot(\mathscr{X}|\mathscr{Y})}^{\bm{a}\cdot\overline{\pmb{\mathscr{L}}}+(b-k)\overline{\mathscr{A}}}
\end{align*}
on $H^0_{\mathscr{X}|\mathscr{Y}}(\bm{a}\cdot\pmb{\mathscr{L}}+(b-k)\mathscr{A};\mathcal{V})$, where $\|\cdot\|_{\sup,\quot(\mathscr{X}|(k+1)\mathscr{Y}),\sub(\sigma_k)}^{\bm{a}\cdot\overline{\pmb{\mathscr{L}}}+b\overline{\mathscr{A}}}$ is the subspace norm induced from
\[
\left(H^0_{\mathscr{X}|(k+1)\mathscr{Y}}(\bm{a}\cdot\pmb{\mathscr{L}}+b\mathscr{A};\mathcal{V}),\|\cdot\|_{\sup,\quot(\mathscr{X}|(k+1)\mathscr{Y})}^{\bm{a}\cdot\overline{\pmb{\mathscr{L}}}+b\overline{\mathscr{A}}}\right)
\]
via $\sigma_k$, and $\|\cdot\|_{\sup,\sub(s_0^{\otimes k}),\quot(\mathscr{X}|\mathscr{Y})}^{\bm{a}\cdot\overline{\pmb{\mathscr{L}}}+b\overline{\mathscr{A}}}$ is the quotient norm induced from
\[
\left(H^0(\bm{a}\cdot\pmb{\mathscr{L}}+(b-k)\mathscr{A};\mathcal{V}),\|\cdot\|_{\sup,\sub(s_0^{\otimes k})}^{\bm{a}\cdot\overline{\pmb{\mathscr{L}}}+\overline{\mathscr{A}}}\right).
\]
Hence, by \eqref{eqn:elementary:property:of:ah:rescaling}, \eqref{eqn:elementary:property:of:ah:norm:increase}, \eqref{eqn:Snappers:theorem}, and Claim~\ref{clm:Snapper:and:aSnapper}, one gets a constant $C>0$ such that
\begin{multline}
\label{eqn:fund:est:Berstein:Markov}
\ahs{\quot(\mathscr{X}|(k+1)\mathscr{Y}),\sub(\sigma_k)}\left(\bm{a}\cdot\overline{\pmb{\mathscr{L}}}+(b-k)\overline{\mathscr{A}};\mathcal{V}\right) \\
\leq \ahs{\quot(\mathscr{X}|\mathscr{Y})}\left(\bm{a}\cdot\overline{\pmb{\mathscr{L}}}+(b-k)\overline{\mathscr{A}};\mathcal{V}\right) + C(\|\bm{a}\|_1+b)^{d} \leq C(\|\bm{a}\|_1+b)^{d}
\end{multline}
for $\bm{a}\in\ZZ_{\geq 0}^r$ and $b,k\in\ZZ_{>0}$ with $\|\bm{a}\|_1+b>0$ and $k\leq b$.

By \eqref{eqn:fund_est_2} and \eqref{eqn:fund:est:Berstein:Markov},
\begin{multline}
\label{eqn:fund_est_3}
 \ahs{\quot(\mathscr{X}|(k+1)\mathscr{Y})}\left(\bm{a}\cdot\overline{\pmb{\mathscr{L}}}+b\overline{\mathscr{A}};\mathcal{V}\right) \leq\ahs{\quot(\mathscr{X}|k\mathscr{Y})}\left(\bm{a}\cdot\overline{\pmb{\mathscr{L}}}+b\overline{\mathscr{A}};\mathcal{V}\right) \\
 +C(\|\bm{a}\|_1+b)^{d}.
\end{multline}
By summing up \eqref{eqn:fund_est_3} for $k=1,2,\dots,b-1$ and by using Claim~\ref{clm:Snapper:and:aSnapper} again, one has
\begin{align*}
\ahs{\quot(\mathscr{X}|b\mathscr{Y})}\left(\bm{a}\cdot\overline{\pmb{\mathscr{L}}}+b\overline{\mathscr{A}};\mathcal{V}\right) & \leq\ahs{\quot(\mathscr{X}|\mathscr{Y})}\left(\bm{a}\cdot\overline{\pmb{\mathscr{L}}}+b\overline{\mathscr{A}};\mathcal{V}\right)+C(\|\bm{a}\|_1+b)^{d}b \\
& \leq C(\|\bm{a}\|_1+b)^{d}b.
\end{align*}
Therefore, by \eqref{eqn:fund_est_1}, one obtains
\[
\ah\left(\bm{a}\cdot\overline{\pmb{\mathscr{L}}}+b\overline{\mathscr{A}};\mathcal{V}\right)\leq\ah\left(\bm{a}\cdot\overline{\pmb{\mathscr{L}}};\mathcal{V}\right)+C\left((\|\bm{a}\|_1+b)^{d}b+\|\bm{a}\|_1^d\log\|\bm{a}\|_1\right)
\]
for every $\bm{a}\in\ZZ_{\geq 0}^r$ with $\|\bm{a}\|_1>0$ and $b\in\ZZ_{>0}$.
\end{Steps}
\end{proof}

\section{Arithmetic volumes of $\ell^1$-adelic $\RR$-Cartier divisors}\label{sec:ell1adelic}

\subsection{Preliminaries}
\label{subsec:prelim}

In this section, we recall definitions and basic properties of adelically normed vector spaces (section~\ref{subsubsec:prelim:adelically:normed}), Berkovich analytic spaces (section~\ref{subsubsec:Berkovich:analytic:space}), and adelic Green functions (section~\ref{subsubsec:model:functions}).

\subsubsection{}\label{subsubsec:prelim:adelically:normed}

Let $K$ be a number field.
Let $\overline{V}:=\left(V,(\|\cdot\|_v^{\overline{V}})_{v\in M_K}\right)$ be a couple of a finite-dimensional $K$-vector space $V$ and a collection $(\|\cdot\|_v^{\overline{V}})_{v\in M_K}$ such that each $\|\cdot\|_v^{\overline{V}}$ is a $(K_v,|\cdot|_v)$-norm on $V_{K_v}=V\otimes_KK_v$ and such that, if $v\in M_K^{\rm fin}$, then $\|\cdot\|_v^{\overline{V}}$ is non-Archimedean.
For such a $\overline{V}$, we set
\begin{equation}
\aHzf(\overline{V}):=\left\{s\in V\,:\,\text{$\|s\|_v^{\overline{V}}\leq 1$ for every $v\in M_K^{\rm fin}$}\right\},
\end{equation}
\begin{equation}
\aHzsm(\overline{V}):=\left\{s\in\aHzf(\overline{V})\,:\,\|s\|_{\infty}^{\overline{V}}\leq 1\right\},\quad
\aHz(\overline{V}):=\left\{s\in\aHzsm(\overline{V})\,:\,\|s\|_{\infty}^{\overline{V}}<1\right\},
\end{equation}
and $\ah(\overline{V}):=\log\sharp\aHzs{\ast}(\overline{V})$ for $\ast=\text{s and ss}$.
Note that $\aHzf(\overline{V})$ is a $O_K$-submodule of $V$ and $\ah(\overline{V})$ may be infinite.

We set, for $\lambda\in\RR$,
\[
\mathcal{F}^{\lambda}(\overline{V}):=\left\langle s\in\aHzf(\overline{V})\,:\,\|s\|_{\infty}^{\overline{V}}\leq e^{-\lambda}\right\rangle_K
\]
(see Notation and terminology~\ref{subsubsec:NT:R-modules}), and set
\begin{equation}
\label{eqn:definition:of:e:max}
e_{\max}(\overline{V}):=\sup\left\{\lambda\in\RR\,:\,\mathcal{F}^{\lambda}(\overline{V})\neq 0\right\}.
\end{equation}

\begin{proposition}
\label{prop:definition:of:adelically:normed:vector:spaces}
Let $\overline{V}=\left(V,(\|\cdot\|_v^{\overline{V}})_{v\in M_K}\right)$ be a couple of a finite-dimensional $K$-vector space $V$ and a collection $(\|\cdot\|_v^{\overline{V}})_{v\in M_K}$ such that each $\|\cdot\|_v^{\overline{V}}$ is a $(K_v,|\cdot|_v)$-norm on $V_{K_v}$ and such that, if $v\in M_K^{\rm fin}$, then $\|\cdot\|_v^{\overline{V}}$ is non-Archimedean.
\begin{enumerate}
\item The following are equivalent.
\begin{enumerate}
\item For each $s\in V$, $\|s\|_v^{\overline{V}}\leq 1$ for all but finitely many $v\in M_K$.
\item $\aHzf(\overline{V})$ contains an $O_K$-submodule $E$ of $V$ satisfying $E_K=V$.
\end{enumerate}
\item Suppose that $\overline{V}$ satisfies the equivalent conditions of (1).
The following are then equivalent.
\begin{enumerate}
\item $\aHzf(\overline{V})$ is a finitely generated $O_K$-module.
\item $\aHzsm(\overline{V})$ is finite.
\item $\aHz(\overline{V})$ is finite.
\item $e_{\max}(\overline{V})<+\infty$.
\end{enumerate}
\end{enumerate}
\end{proposition}

\begin{proof}
(1) (a) $\Rightarrow$ (b): For each $s\in V$, one can find an $n\geq 1$ such that $\|ns\|_v^{\overline{V}}=|n|_v\|s\|_v^{\overline{V}}\leq 1$ for every $v\in M_K$ by the condition (a).
Thus $ns\in\aHzf(\overline{V})$, which implies $V=\aHzf(\overline{V})_K$.

(b) $\Rightarrow$ (a): For each $s\in V$, there exists an $\alpha\in O_K$ such that $\alpha s\in E$.
Hence $\|\alpha s\|_v^{\overline{V}}=\|s\|_v^{\overline{V}}\leq 1$ for all but finitely many $v\in M_K$.

For the assertion (2), we refer to \cite[Proposition~2.4]{Boucksom_Chen} and \cite[Proposition~C.2.4]{Bombieri_Gubler}.
\end{proof}

\begin{definition}
An \emph{adelically normed $K$-vector space} is a couple $\left(V,(\|\cdot\|_v^{\overline{V}})_{v\in M_K}\right)$ satisfying the all conditions in Proposition~\ref{prop:definition:of:adelically:normed:vector:spaces}(1),(2).
Notice that here the existence of an $O_K$-model of $\overline{V}$ that defines $\|\cdot\|_v^{\overline{V}}$ except for finitely many $v$ is not assumed while it is in the classical definition in \cite[(1.6)]{ZhangAdelic} and in \cite[Definition~3.1]{Gaudron07}.

Let $\lambda\in\RR$ and let $v\in M_K$.
We define an adelically normed $K$-vector space $\overline{V}(\lambda[v])=\left(V,(\|\cdot\|_w^{\overline{V}(\lambda[v])})_{w\in M_K}\right)$ as
\begin{equation}
\|\cdot\|_w^{\overline{V}(\lambda[v])}:=\begin{cases}\|\cdot\|_w^{\overline{V}} & \text{if $w\neq v$ and} \\ e^{-\lambda}\|\cdot\|_v^{\overline{V}} & \text{if $w=v$.} \end{cases}
\end{equation}
\end{definition}

\begin{lemma}
\label{lem:adelic:rescaling:adelically:normed:vector:spaces}
Let $\lambda\in\RR_{\geq 0}$ and let $v\in M_K^{\rm fin}$.
If we set $p\ZZ:=\mathfrak{p}_v\cap\ZZ$, then
\[
0 \leq \ah\left(\overline{V}(\lambda[v])\right)-\ah\left(\overline{V}\right) \leq \left(\left\lceil\frac{\lambda}{-\log|p|_v}\right\rceil\log(p)+2\right)\dim_{\QQ}V.
\]
\end{lemma}

\begin{proof}
Set
\[
n_{\lambda}:=\left\lceil\frac{\lambda}{-\log|p|_v}\right\rceil.
\]
We are going to show
\begin{equation}
\label{eqn:adelic:rescaling:adelically:normed:vector:spaces:1}
p^{n_{\lambda}}\aHzf\left(\overline{V}(\lambda[v])\right)\subset\aHzf\left(\overline{V}\right).
\end{equation}
Suppose that $s\in \aHzf\left(\overline{V}(\lambda[v])\right)$.
Then
\[
\|p^{n_{\lambda}}s\|_w^{\overline{V}}=|p|_w^{n_{\lambda}}\|s\|_w^{\overline{V}}\leq\begin{cases} 1 & \text{if $w\neq v$ and} \\ e^{\lambda}|p|_v^{n_{\lambda}}  & \text{if $w=v$.}\end{cases}
\]
Since $\lambda+n_{\lambda}\log|p|_v\leq 0$, we have $p^{n_{\lambda}}s\in\aHzf\left(\overline{V}\right)$.

We apply \eqref{eqn:elementary:property:of:ah:rescaling:2} to the inclusion $\aHzf\left(\overline{V}\right)\subset\aHzf\left(\overline{V}(\lambda[v])\right)$, and obtain
\begin{align*}
\ah\left(\overline{V}(\lambda[v])\right) & \leq\ah\left(\overline{V}\right)+\log\sharp\left(\aHzf\left(\overline{V}(\lambda[v])\right)/\aHzf\left(\overline{V}\right)\right)+2\dim_{\QQ}V \\
& \leq\ah\left(\overline{V}\right)+\log\sharp\left(\aHzf\left(\overline{V}(\lambda[v])\right)/p^{n_{\lambda}}\aHzf\left(\overline{V}(\lambda[v])\right)\right)+2\dim_{\QQ}V \\
&\leq\ah\left(\overline{V}\right)+\left(n_{\lambda}\log(p)+2\right)\dim_{\QQ}V
\end{align*}
by \eqref{eqn:adelic:rescaling:adelically:normed:vector:spaces:1}.
\end{proof}

\subsubsection{}
\label{subsubsec:Berkovich:analytic:space}
Let $X$ be a normal, projective, and geometrically connected $K$-variety.
For $v=\infty$, we denote by $X_{\infty}^{\rm an}$ the complex analytic space associated to $X_{\CC}:=X\times_{\Spec(\QQ)}\Spec(\CC)$.
For $v\in M_K^{\rm fin}$, we denote by $(X_v^{\rm an},\rho_v:X_v^{\rm an}\to X_{K_v})$ the Berkovich analytic space associated to $X_{K_v}$ (see \cite{BerkovichBook}).
For each $x\in X_v^{\rm an}$, we denote by $\kappa(x)$ the residue field of $\rho_v(x)\in X_{K_v}$ and by $|\cdot|_x$ the corresponding norm on $\kappa(x)$.
Given a local function $f$ on $X_{K_v}$ defined around $\rho_v(x)$, we write
\begin{equation}
|f|(x):=|f(\rho_v(x))|_x.
\end{equation}

An \emph{$O_K$-model} of $X$ is a reduced, irreducible, projective, and flat $O_K$-scheme with generic fiber $\mathscr{X}_{K}\simeq X$.
Given an $O_K$-model $\mathscr{X}$ of $X$, we set
\begin{equation}
\widetilde{\mathscr{X}}_v:=\mathscr{X}\times_{\Spec(O_K)}\Spec(\widetilde{K}_v).
\end{equation}
For each $x\in X_v^{\rm an}$, the morphism $\rho_v(x):\Spec(\kappa(x))\to\mathscr{X}_{K_v^{\circ}}$ uniquely extends to a morphism $\Spec(\kappa(x)^{\circ})\to\mathscr{X}_{K_v^{\circ}}$ by the valuative criterion of properness.
We define $r_v^{\mathscr{X}}(x)$ as the image of the closed point of $\Spec(\kappa(x)^{\circ})$.

Let $\mathscr{U}=\Spec(\mathscr{A})$ be an affine open subscheme of $\mathscr{X}_{K_v^{\circ}}$ with $\mathscr{U}\cap\widetilde{\mathscr{X}}_v\neq\emptyset$, and set $U=\mathscr{U}_{K_v}=\Spec(A)$.
We put
\begin{equation}
U_{v,\mathscr{U}}^{\rm an}:=\left\{x\in U_v^{\rm an}\,:\,\text{$|f|(x)\leq 1$ for all $f\in\mathscr{A}$}\right\}.
\end{equation}

\begin{lemma}
\begin{enumerate}
\item $U_{v,\mathscr{U}}^{\rm an}=\left(r_v^{\mathscr{X}}\right)^{-1}\left(\mathscr{U}\cap\widetilde{\mathscr{X}}_v\right)$.
\item $U_{v,\mathscr{U}}^{\rm an}$ is compact.
\end{enumerate}
\end{lemma}

\begin{proof}
(1): If $x\in U_{v,\mathscr{U}}^{\rm an}$, then the image of the homomorphism $\mathscr{A}\to\kappa(x)$ is in $\kappa(x)^{\circ}$, so $r_v^{\mathscr{X}}(x)\in\mathscr{U}$.
Conversely, if $r_v^{\mathscr{X}}(x)\in\mathscr{U}\cap\widetilde{\mathscr{X}}_v$, then $\rho_v(x)\in U$ and $x\in\rho_v^{-1}(U)=U_v^{\rm an}$.
Since $r_v^{\mathscr{X}}(x)\in\mathscr{U}$, the image of the morphism $\Spec(\kappa(x)^{\circ})\to\mathscr{X}_{K_v^{\circ}}$ is in $\mathscr{U}$, so $f(\rho_v(x))\in\kappa(x)^{\circ}$ for every $f\in\mathscr{A}$.

(2): The map
\[
u:U_{v}^{\rm an}\to I:=\prod_{f\in\mathscr{A}}\RR_{\geq 0},\quad x\mapsto (|f|(x))_{f\in\mathscr{A}},
\]
is injective and continuous, where $I$ is endowed with the product topology.
By Tychonoff's theorem, $J:=\prod_{f\in\mathscr{A}}[0,1]$ is a compact subset of $I$, and $U_{v,\mathscr{U}}^{\rm an}=u^{-1}(J)$.
Thus it suffices to show that $u$ is a closed map.
Suppose that $\left(u(x_{\alpha})\right)_{\alpha}$ is a net in $I$ that converges to $\left(\lambda_f\right)_{f\in\mathscr{A}}\in I$.
For each $f\in\mathscr{A}$, we set $|f|_x:=\lambda_f$.

\begin{claim}
\label{clm:limit:of:the:net}
$|\cdot|_x$ extends to a multiplicative seminorm on $A$ whose restriction to $K_v$ is $|\cdot|_v$.
\end{claim}

\begin{proof}[{Proof of Claim~\ref{clm:limit:of:the:net}}]
Since, for every $\alpha$, $|\cdot|_{x_{\alpha}}$ satisfies the conditions:
\begin{itemize}
\item $|a|(x_{\alpha})=|a|_v$ for $a\in K_v$,
\item $|f-g|(x_{\alpha})\leq |f|(x_{\alpha})+|g|(x_{\alpha})$ for $f,g\in\mathscr{A}$, and
\item $|fg|(x_{\alpha})=|f|(x_{\alpha})|g|(x_{\alpha})$ for $f,g\in\mathscr{A}$,
\end{itemize}
we know that the limit $|\cdot|_x$ is a multiplicative seminorm on $\mathscr{A}$.
For a general $f\in A$, we can take an $n\geq 0$ such that $\varpi_v^nf\in\mathscr{A}$, and define
\[
|f|_x:=|\varpi_v|_v^{-n}|\varpi_v^nf|_x,
\]
which does not depend on a specific choice of $n\geq 0$.
Then $|\cdot|_x$ is a multiplicative seminorm on $A$.
\end{proof}

By Claim~\ref{clm:limit:of:the:net}, $|\cdot|_x$ corresponds to a point $x\in U_v^{\rm an}$.
Since $|f|_{x_{\alpha}}\to |f|_x$ for every $f\in A$, the net $(x_{\alpha})_{\alpha}$ converges to $x$ in the Gel\textprime fand topology, and $(\lambda_f)_{f\in\mathscr{A}}=u(x)$.
It implies that $u$ is a closed map.
\end{proof}

Let $\widetilde{\mathscr{X}}_{v,\rm gen}$ be the set of all the generic points of irreducible components of $\widetilde{\mathscr{X}}_v$.
For each $\xi\in\widetilde{\mathscr{X}}_{v,{\rm gen}}$, $\left(r_v^{\mathscr{X}}\right)^{-1}(\xi)$ consists of a single point $x_{\xi}$ given as
\begin{equation}
|\phi|_{x_{\xi}}:=\left(\sharp\widetilde{K}_v\right)^{-\frac{\ord_{\xi}(\phi)}{\ord_{\xi}(\varpi_v)}}
\end{equation}
for $\phi\in\Rat(X)$.
We set $\Gamma(X_v^{\rm an}):=\left\{x_{\xi}\,:\,\xi\in\widetilde{\mathscr{X}}_{v,{\rm gen}}\right\}$ (see also \cite[Proposition~2.4.4 and Corollary~2.4.5]{BerkovichBook}).

\begin{lemma}
\label{lem:maximum:modulus:principle}
Suppose that $\mathscr{A}$ is integrally closed in $A$.
Then, for each $f\in A$,
\[
\max_{x\in U_{v,\mathscr{U}}^{\rm an}}\left\{|f|(x)\right\}=\max_{x\in\Gamma(X_v^{\rm an})\cap U_{v,\mathscr{U}}^{\rm an}}\left\{|f|(x)\right\}.
\]
\end{lemma}

\begin{proof}
Since $\mathscr{U}\cap\widetilde{\mathscr{X}}_v\neq\emptyset$, we have $\Gamma(X_v^{\rm an})\cap U_{v,\mathscr{U}}^{\rm an}\neq\emptyset$.
The inequality $\geq$ is obvious, so that we are going to show the reverse.
Choose a $\xi_0\in\widetilde{\mathscr{X}}_{v,{\rm gen}}$ such that
\[
|f|(x_{\xi_0})=\max_{x\in\Gamma(X_v^{\rm an})\cap U_{v,\mathscr{U}}^{\rm an}}\left\{|f|(x)\right\}.
\]
If we set $n:=\ord_{\xi_0}(\varpi_v)$ and $l:=\ord_{\xi_0}(f)$, then $\ord_{\xi}(\varpi_v^{-l}f^n)\geq 0$ for every $\xi\in\widetilde{\mathscr{X}}_{v,{\rm gen}}$.
By \cite[Lemma~2.3(3)]{IkomaRem}, it implies $\varpi_v^{-l}f^n\in\mathscr{A}$.
Hence
\[
|f|(x)\leq |\varpi_v|_v^{\frac{l}{n}}=|f|(x_{\xi_0})
\]
for every $x\in U_{v,\mathscr{U}}^{\rm an}$.
\end{proof}

\subsubsection{}
\label{subsubsec:model:functions}
Let $X$ be a normal, projective, and geometrically connected $K$-variety, let $\KK=\text{$\RR$, $\QQ$, or $\ZZ$}$, and let $D$ be a $\KK$-Cartier divisor on $X$.
The \emph{support} of $D$ is a Zariski closed subset defined as
\begin{equation}
\Supp(D):=\bigcup_{\substack{\text{$Z$: prime Weil divisor,} \\ \ord_Z(D)\neq 0}}Z
\end{equation}
(see \cite[Notation and terminology~2]{IkomaDiff1}).
Let $v\in M_K$.
A \emph{$D$-Green function} on $X_v^{\rm an}$ is a continuous map $g_v:(X\setminus\Supp(D))_v^{\rm an}\to\RR$ such that, for each $x\in X_v^{\rm an}$,
\begin{equation}
g_v(x)+\log|f|^2(x)
\end{equation}
extends to a continuous function around $x$, where $f\in\Rat(X)^{\times}\otimes_{\ZZ}\KK$ is a local equation defining $D$ around $\rho_v(x)$ (see \cite[Definition~2.1.1]{MoriwakiAdelic}).
If $v=\infty$, we assume that a $D$-Green function is invariant under the complex conjugation map.
We then set
\begin{equation}
\aDiv_{\KK}^{\rm tot}(X):=\left\{\left(D,\sum_{v\in M_K}g_v^{\overline{D}}[v]\right)\,:\,\begin{array}{l} \text{$D\in\Div_{\KK}(X)$ and $g_v^{\overline{D}}$ is a $D$-Green} \\ \text{function on $X_v^{\rm an}$ for each $v\in M_K$}\end{array}\right\}.
\end{equation}

An element $\overline{D}\in\aDiv_{\RR}^{\rm tot}(X)$ is called \emph{effective} if
\begin{equation}
D\geq 0\quad\text{and}\quad \essinf_{x\in X_v^{\rm an}}g_v^{\overline{D}}(x)\geq 0,\quad\forall v\in M_K,
\end{equation}
and, for $\overline{D},\overline{E}\in\aDiv_{\RR}^{\rm tot}(X)$, we write $\overline{D}\leq\overline{E}$ if $\overline{E}-\overline{D}$ is effective.
Each $g_v^{\overline{D}}$ defines the supremum norm on $H^0(D)$ as
\begin{equation}
\|\phi\|_{v,\sup}^{\overline{D}}:=\sup_{x\in X_v^{\rm an}}|\phi|(x)\exp\left(\frac{1}{2}g_v^{\overline{D}}(x)\right)
\end{equation}
for $\phi\in H^0(D)$ (see \cite[Proposition~2.1.3]{MoriwakiAdelic}).

In the following, we impose on $\nu\in\DV(\Rat(X))$ a condition that the restriction of $\nu$ to $K$ is trivial (see section~\ref{subsec:base:conditions}).
Given a $\overline{D}\in\aDiv_{\RR}^{\rm tot}(X)$ and a $\mathcal{V}\in\BC_{\RR}(X)$, we set
\begin{equation}
\aHzs{\ast}\left(\overline{D};\mathcal{V}\right):=\aHzs{\ast}\left(H^0(D;\mathcal{V}),(\|\cdot\|_{v,\sup}^{\overline{D}})_{v\in M_K}\right)
\end{equation}
for $\ast=\text{f, s, and ss}$, and set $\ah\left(\overline{D};\mathcal{V}\right):=\log\sharp\aHzs{\ast}\left(\overline{D};\mathcal{V}\right)$ for $\ast=\text{s and ss}$ (see section~\ref{subsubsec:prelim:adelically:normed} and \eqref{eqn:ah:for:short:1}).
An \emph{$O_K$-model} of a couple $(X,D)$ is a couple $(\mathscr{X},\mathscr{D})$ such that $\mathscr{X}$ is a normal $O_K$-model of $X$ and such that $\mathscr{D}$ is an $\RR$-Cartier divisor on $\mathscr{X}$ satisfying $\mathscr{D}|_X=D$.
Given an $O_K$-model $(\mathscr{X},\mathscr{D})$ of $(X,D)$ and a $v\in M_K^{\rm fin}$, we define the $D$-Green function associated to $(\mathscr{X},\mathscr{D})$ as
\begin{equation}
 g_v^{(\mathscr{X},\mathscr{D})}(x):=-\log |f'|^2(x),
\end{equation}
where $f'$ is a local equation defining $\mathscr{D}$ around $r_v^{\mathscr{X}}(x)$.

Let $\KK:=\text{$\RR$, $\QQ$, or $\ZZ$}$.
A couple $\overline{\mathscr{D}}=(\mathscr{D},g_{\infty}^{\overline{\mathscr{D}}})$ on $\mathscr{X}$ such that $(\mathscr{X},\mathscr{D})$ is an $O_K$-model of $(X,D)$ with $\mathscr{D}\in\Div_{\KK}(\mathscr{X})$ and such that $g_{\infty}^{\overline{\mathscr{D}}}$ is a $D$-Green function on $X_{\infty}^{\rm an}$ is called an \emph{arithmetic $\KK$-Cartier divisor} on $\mathscr{X}$.
If $X$ is smooth and $g_{\infty}^{\overline{\mathscr{D}}}$ is of $C^{\infty}$-type, then $\overline{\mathscr{D}}$ is said to be \emph{of $C^{\infty}$-type} (see \cite[section~2.3]{MoriwakiZar}).
We denote by $\aDiv_{\KK}(\mathscr{X})$ (respectively, $\aDiv_{\KK}(\mathscr{X};C^{\infty})$) the $\KK$-module of all the arithmetic $\KK$-Cartier divisors (respectively, arithmetic $\KK$-Cartier divisors of $C^{\infty}$-type) on $\mathscr{X}$.
If $\KK=\ZZ$ and $\text{--- $=$ a blank or $C^{\infty}$}$, we will abbreviate $\aDiv(\mathscr{X};\text{---}):=\aDiv_{\ZZ}(\mathscr{X};\text{---})$ as usual.

Given a couple $(\overline{\mathscr{D}};\mathcal{V})$ of a $\overline{\mathscr{D}}\in\aDiv_{\RR}(\mathscr{X})$ and a $\mathcal{V}\in\BC_{\RR}(\mathscr{X})$, we abbreviate
\begin{equation}
\ah\left(\overline{\mathscr{D}};\mathcal{V}\right):=\ah\left(H^0(\overline{\mathscr{D}};\mathcal{V}),\|\cdot\|_{\infty,\sup}^{\overline{\mathscr{D}}}\right)
\end{equation}
for $\ast=\text{s and ss}$ (see Notation and terminology~\ref{subsubsec:NT:normed:Z-modules} and \eqref{eqn:ah:for:short:1}), and define
\begin{equation}
\avol\left(\overline{\mathscr{D}};\mathcal{V}\right):=\limsup_{\substack{m\in\ZZ, \\ m\to+\infty}}\frac{\als\left(m\overline{\mathscr{D}};m\mathcal{V}\right)}{m^{\dim\mathscr{X}}/(\dim\mathscr{X})!}.
\end{equation}
Moreover, the \emph{adelization} of $\overline{\mathscr{D}}\in\aDiv_{\RR}(\mathscr{X})$ is defined as
\begin{equation}
\overline{\mathscr{D}}^{\rm ad}:=\left(D,\sum_{v\in M_K^{\rm fin}}g_v^{(\mathscr{X},\mathscr{D})}[v]+g_{\infty}[\infty]\right),
\end{equation}
which belongs to $\aDiv_{\RR}^{\rm tot}(X)$.

\subsection{The space of continuous functions}
\label{subsec:space:of:continuous:functions}

Let $K$ be a number field, and let $X$ be a projective and geometrically connected $K$-variety.
For each $v\in M_K$, we denote by $C(X_v^{\rm an})$ the space of $\RR$-valued continuous functions on $X_v^{\rm an}$ that are assumed to be invariant under the complex conjugation if $v=\infty$.
We endow $C(X_v^{\rm an})$ with the supremum norm:
\[
\|f\|_{\sup}:=\sup_{x\in X_v^{\rm an}}|f(x)|_{\infty}
\]
for $f\in C(X_v^{\rm an})$, where $|f(x)|_{\infty}$ denotes the usual absolute value of the real number $f(x)$ (see Notation and terminology~\ref{subsubsec:NT:number:fields}).
By elementary arguments, $(C(X_v^{\rm an}),\|\cdot\|_{\sup})$ is a Banach algebra for every $v\in M_K$.
We denote by
\begin{equation}
C_{\rm tot}(X):=\prod_{v\in M_K}C(X_v^{\rm an})=\left\{\bm{f}=\sum_{v\in M_K}f_v[v]\,:\,f_v\in C(X_v^{\rm an})\right\}
\end{equation}
the algebraic direct product of the family $\left(C(X_v^{\rm an})\right)_{v\in M_K}$, and by
\begin{equation}
C(X):=\bigoplus_{v\in M_K}C(X_v^{\rm an})
\end{equation}
the algebraic direct sum of $\left(C(X_v^{\rm an})\right)_{v\in M_K}$.
The \emph{$\ell^1$-norm} of an $\bm{f}\in C_{\rm tot}(X)$ is
\begin{equation}
\|\bm{f}\|_{\ell^1}:=\sum_{v\in M_K}\|f_v\|_{\sup},
\end{equation}
where the sum is taken with respect to the net indexed by all the finite subsets of $M_K$, and the \emph{$\ell^1$-direct sum} of $(C(X_v^{\rm an}))_{v\in M_K}$ is given as
\[
\aC(X):=\left\{\bm{f}=(f_v)_{v\in M_K}\,:\,\|\bm{f}\|_{\ell^1}<+\infty\right\}
\]
endowed with the $\ell^1$-norm.
For $\bm{f},\bm{g}\in C_{\rm tot}(X)$, we write $\bm{f}\leq\bm{g}$ if $f_v\leq g_v$ for every $v\in M_K$.
If $\bm{f},\bm{g}\in\aC(X)$, then the entrywise product $\bm{f}\bm{g}$ satisfies
\[
\|\bm{f}\bm{g}\|_{\ell^1}\leq\sum_{v\in M_K}\|f_v\|_{\sup}\|g_v\|_{\sup}\leq \sup_{v\in M_K}\left\{\|f_v\|_{\sup}\right\}\cdot\|\bm{g}\|_{\ell^1}\leq\|\bm{f}\|_{\ell^1}\cdot\|\bm{g}\|_{\ell^1},
\]
so $\bm{f}\bm{g}\in\aC(X)$.
By the same arguments as in \cite[page 67, Theorem~3.11]{RudinTrilogy2}, one verifies that $(\aC(X),\|\cdot\|_{\ell^1})$ is a Banach algebra.
Note that $C_{\rm tot}(\Spec(K))$ is just $\RR^{M_K}$ and $\aC(\Spec(K))=\ell^1(M_K)$ is just the $\ell^1$-sequence space indexed by $M_K$.
We will identify $C_{\rm tot}(\Spec(K))$ with the space of constant functions in $C_{\rm tot}(X)$.

\begin{lemma}
\label{lem:approximation:theorem:continuous:functions}
Let $\bm{f}\in\aC(X)$.
Given any $\varepsilon>0$, there exists a $\bm{h}\in C(X)$ such that
\[
\bm{h}\leq \bm{f}\quad\text{and}\quad\|\bm{f}-\bm{h}\|_{\ell^1}\leq\varepsilon.
\]
\end{lemma}

\begin{proof}
Since $\sum_{v\in M_K}\|f_v\|_{\sup}<+\infty$, there is a finite subset $S\subset M_K$ such that
\[
\sum_{v\in (M_K\setminus S)}\|f_v\|_{\sup}\leq\varepsilon.
\]
Hence $\bm{h}:=\sum_{v\in S}f_v[v]$ satisfies the required conditions.
\end{proof}

\subsection{$\ell^1$-adelic $\RR$-Cartier divisors}
\label{subsec:definition:of:ell1:adelic}

Let $X$ be a normal, projective, and geometrically connected $K$-variety.
The natural homomorphisms
\begin{equation}
C_{\rm tot}(X)\to \aDiv_{\RR}^{\rm tot}(X),\quad \bm{f}\mapsto (0,\bm{f}),
\end{equation}
and
\begin{equation}
\label{eqn:definition:of:zeta}
\zeta:\aDiv_{\RR}^{\rm tot}(X)\to\Div_{\RR}(X),\quad \overline{D}=\left(D,\sum_{v\in M_K}g_v^{\overline{D}}[v]\right)\mapsto \zeta(\overline{D})=D,
\end{equation}
form an exact sequence
\begin{equation}
0\to C_{\rm tot}(X)\to\aDiv_{\RR}^{\rm tot}(X)\xrightarrow{\zeta}\Div_{\RR}(X)\to 0.
\end{equation}

Let $\KK$ and $\KK'$ be either $\RR$, $\QQ$, or $\ZZ$.
Given a $\overline{D}\in\aDiv_{\KK}^{\rm tot}(X)$, we set
\begin{equation}
\Mod_{\KK'}(\overline{D}):=\left\{\left(\mathscr{X},(\mathscr{D},g_{\infty})\right)\,:\,\begin{array}{l} \text{$(\mathscr{X},\mathscr{D})$ is an $O_K$-model of $(X,D)$, $(\mathscr{D},g_{\infty})$} \\ \text{$\in\aDiv_{\KK'}(\mathscr{X})$, and $\overline{\mathscr{D}}^{\rm ad}\leq\overline{D}$} \end{array}\right\}.
\end{equation}
We call $\overline{D}\in\aDiv_{\KK}^{\rm tot}(X)$ an \emph{adelic $\KK$-Cartier divisor} if there exists an $(\mathscr{X},\overline{\mathscr{D}})\in\Mod_{\RR}(\overline{D})$ such that $\overline{D}-\overline{\mathscr{D}}^{\rm ad}\in C(X)$.
Denote by $\aDiv_{\KK}(X)$ the $\KK$-module of all the adelic $\KK$-Cartier divisors on $X$.
As before, we will write $\aDiv(X):=\aDiv_{\ZZ}(X)$.
For a $\overline{D}\in\aDiv_{\RR}(X)$, there are a nonempty open subset $U$ of $\Spec(O_K)$ and an $(\mathscr{X},\mathscr{D})\in\Mod_{\RR}(\overline{D})$ such that $g_v^{\overline{D}}=g_v^{(\mathscr{X},\mathscr{D})}$ for every $v\in U$.
In this case, we call the couple $(\mathscr{X}_U,\mathscr{D}_U)$ a \emph{$U$-model of definition} for $\overline{D}$ (see \cite[Definition~4.1.1]{MoriwakiAdelic} and \cite[Notation and terminology~4]{IkomaDiff1}).
Given a $\overline{D}\in\aDiv_{\RR}(X)$ and a $\mathcal{V}\in\BC_{\RR}(X)$, we define
\begin{equation}
\avol\left(\overline{D};\mathcal{V}\right):=\limsup_{\substack{m\in\ZZ, \\ m\to+\infty}}\frac{\als\left(m\overline{D};m\mathcal{V}\right)}{m^{\dim X+1}/(\dim X+1)!},
\end{equation}
which is finite as in \cite[section~2.5]{IkomaDiff1}.

\begin{proposition}
\label{prop:definition:of:ell1:adelic}
Let $\KK=\text{$\RR$ or $\QQ$}$.
For any $\overline{D}\in\aDiv_{\KK}^{\rm tot}(X)$, the following are equivalent.
\begin{enumerate}
\item There exists an $(\mathscr{X},\overline{\mathscr{D}})\in\Mod_{\KK}(\overline{D})$ such that $\left\|\overline{D}-\overline{\mathscr{D}}^{\rm ad}\right\|_{\ell^1}<+\infty$.
\item For any $(\mathscr{X},\overline{\mathscr{D}})\in\Mod_{\RR}(\overline{D})$, $\left\|\overline{D}-\overline{\mathscr{D}}^{\rm ad}\right\|_{\ell^1}<+\infty$.
\item For any $\varepsilon>0$, there exists an $(\mathscr{X}_{\varepsilon},\overline{\mathscr{D}}_{\varepsilon})\in\Mod_{\KK}(\overline{D})$ such that $\left\|\overline{D}-\overline{\mathscr{D}}_{\varepsilon}^{\rm ad}\right\|_{\ell^1}\leq\varepsilon$.
\item For any $\varepsilon>0$, there exists an $\overline{D}_{\varepsilon}\in\aDiv_{\KK}(X)$ such that $\zeta(\overline{D}_{\varepsilon})=\zeta(\overline{D})$, $\overline{D}_{\varepsilon}\leq\overline{D}$, and $\left\|\overline{D}-\overline{D}_{\varepsilon}\right\|_{\ell^1}\leq\varepsilon$.
\end{enumerate}
\end{proposition}

\begin{proof}
The implications (2) $\Rightarrow$ (1) and (3) $\Rightarrow$ (1) are obvious.
The equivalence (3) $\Leftrightarrow$ (4) results from the approximation theorem (see \cite[Theorem~4.1.3]{MoriwakiAdelic}).

(1) $\Rightarrow$ (2): It suffices to show that for any $(\mathscr{X}',\overline{\mathscr{D}}')\in\Mod_{\RR}(\overline{D})$
\begin{equation}
\label{eqn:definition:of:ell1:adelic:1}
\left\|\overline{\mathscr{D}}^{\rm ad}-\overline{\mathscr{D}}^{\prime{\rm ad}}\right\|_{\ell^1}<+\infty.
\end{equation}
Let $\mathscr{X}''$ be a normal $O_K$-model of $X$ that dominates both $\mathscr{X}$ and $\mathscr{X}'$.
Let $\mu:\mathscr{X}''\to \mathscr{X}$ and $\mu':\mathscr{X}''\to\mathscr{X}'$ be the dominant morphisms.
Then
\[
\left(\mu^*\overline{\mathscr{D}}\right)^{\rm ad}=\overline{\mathscr{D}}^{\rm ad}\quad\text{and}\quad\left(\mu^{\prime *}\overline{\mathscr{D}}\right)^{\prime{\rm ad}}=\overline{\mathscr{D}}^{\prime{\rm ad}}
\]
(see \cite[Proposition~2.1.4]{MoriwakiAdelic}).
We write
\[
D=a_1Z_1+\dots+a_rZ_r
\]
with $a_i\in\RR$ and prime Weil divisors $Z_i$.
Let $\mathscr{Z}_i$ be the Zariski closure of $Z_i$ in $\mathscr{X}''$.
Since
\[
\mu^*\mathscr{D}-\sum_{i=1}^ra_i\mathscr{Z}_i\quad\text{and}\quad\mu^{\prime *}\mathscr{D}^{\prime}-\sum_{i=1}^ra_i\mathscr{Z}_i
\]
are both vertical, one can find a nonempty open subset $U\subset\Spec(O_K)$ such that $\left(\mu^*\mathscr{D}\right)_U=\left(\mu^{\prime *}\mathscr{D}'\right)_U$.
Hence we have \eqref{eqn:definition:of:ell1:adelic:1}.

(1) $\Rightarrow$ (4): Set $(0,\bm{f}):=\overline{D}-\overline{\mathscr{D}}^{\rm ad}$.
By Lemma~\ref{lem:approximation:theorem:continuous:functions}, there exists an $\bm{f}_{\varepsilon}\in C(X)$ such that $\bm{f}_{\varepsilon}\leq\bm{f}$ and such that $\left\|\bm{f}-\bm{f}_{\varepsilon}\right\|_{\ell^1}\leq\varepsilon$.
Set
\[
\overline{D}_{\varepsilon}:=\overline{\mathscr{D}}^{\rm ad}+(0,\bm{f}_{\varepsilon}).
\]
Then $\overline{D}_{\varepsilon}\in\aDiv_{\KK}(X)$, $\overline{D}_{\varepsilon}\leq\overline{D}$, and $\left\|\overline{D}-\overline{D}_{\varepsilon}\right\|_{\ell^1}=\left\|\bm{f}-\bm{f}_{\varepsilon}\right\|_{\ell^1}\leq\varepsilon$.
\end{proof}

\begin{definition}
Let $\KK=\text{$\RR$, $\QQ$, or $\ZZ$}$.
We call an element $\overline{D}\in\aDiv_{\KK}^{\rm tot}(X)$ an \emph{$\ell^1$-adelic $\KK$-Cartier divisor} on $X$ if there exists an $(\mathscr{X},\overline{\mathscr{D}})\in\Mod_{\RR}(\overline{D})$ such that $\left\|\overline{D}-\overline{\mathscr{D}}^{\rm ad}\right\|_{\ell^1}<+\infty$.
We denote by $\aDiv_{\KK}^{\ell^1}(X)$ the $\KK$-module of all the $\ell^1$-adelic $\KK$-Cartier divisors on $X$.
If $\KK=\ZZ$, then the subscript $\ZZ$ will be omitted as usual.

Moreover, we set
\begin{equation}
\aBDiv_{\KK,\RR}^{\ell^1}(X):=\aDiv_{\KK}^{\ell^1}(X)\times\BC_{\RR}(X).
\end{equation}
\end{definition}

Let $\PicSch_{X/K}$ be the Picard scheme of $X$ and let $\PicSch^0_{X/K}$ be the neutral component of $\PicSch_{X/K}$.
Let $\Pic(X)=\PicSch_{X/K}(K)$ be the Picard group of $X$, and let
\begin{equation}
\NS(X):=\PicSch_{X/K}(\overline{K})/\PicSch_{X/K}^0(\overline{K})
\end{equation}
be the N\'eron--Severi group of $X$.
By Severi's theorem of the base, $\NS(X)$ is a finitely generated $\ZZ$-module, and, since $\PicSch^0_{X/K}$ is an abelian variety over $K$ (see for example \cite[Theorem~5.4]{KleimanFGA}), $\PicSch_{X/K}^0(K)$ is also a finitely generated $\ZZ$-module by the Mordell-Weil theorem.
Since
\[
\PicSch_{X/K}^0(\overline{K})\cap\PicSch_{X/K}(K)=\PicSch_{X/K}^0(K),
\]
we obtain an exact sequence
\begin{equation}
0\to\PicSch_{X/K}^0(K)\to\Pic(X)\to\NS(X).
\end{equation}
Hence $\Pic(X)$ is also a finitely generated $\ZZ$-module.

Let $\widehat{P}_{\RR}(X)$ (respectively, $P_{\RR}(X)$) be the $\RR$-subspace of $\aDiv_{\RR}^{\ell^1}(X)$ (respectively, $\Div_{\RR}(X)$) generated by the principal divisors $\widehat{(\phi)}$ (respectively, $(\phi)$) for $\phi\in\Rat(X)^{\times}$.
Let $\Pic_{\RR}(X):=\Pic(X)\otimes_{\ZZ}\RR$ be the $\RR$-vector space of $\RR$-line bundles on $X$.
By \cite[Proposition~II.6.15]{Hart77}, the sequence
\begin{equation}
0\to P_{\RR}(X)\to\Div_{\RR}(X)\xrightarrow{\OO_X}\Pic_{\RR}(X)\to 0
\end{equation}
is exact.
So, if we set
\begin{equation}
\Cl_{\RR}(X):=\Div_{\RR}(X)/P_{\RR}(X),
\end{equation}
then $\Cl_{\RR}(X)=\Pic_{\RR}(X)$ is a finite-dimensional $\RR$-vector space.

\begin{definition}
We define
\[
\aCl_{\RR}^{\ell^1}(X):=\aDiv_{\RR}^{\ell^1}(X)/\widehat{P}_{\RR}(X).
\]
\end{definition}

\begin{lemma}
The sequence
\[
0\to\aC(X)\to\aCl_{\RR}^{\ell^1}(X)\xrightarrow{\zeta}\Cl_{\RR}(X)\to 0
\]
is exact.
\end{lemma}

\begin{proof}
Obviously, the sequence
\[
0\to\aC(X)\to\aDiv_{\RR}^{\ell^1}(X)\xrightarrow{\zeta}\Div_{\RR}(X)\to 0
\]
is exact.
If $\zeta(\overline{D})\in P_{\RR}(X)$, then $D=(\phi)$ for a $\phi\in\Rat(X)^{\times}\otimes_{\ZZ}\RR$ or $D=0$.
Hence $\zeta^{-1}(P_{\RR}(X))=\widehat{P}_{\RR}(X)\oplus\aC(X)$, which infers the required result.
\end{proof}

We fix a section $\iota:\Cl_{\RR}(X)\to\aCl_{\RR}^{\ell^1}(X)$ of $\zeta$ and a norm $\|\cdot\|$ on the finite-dimensional $\RR$-vector space $\Cl_{\RR}(X)$.
We can then define a norm on $\aCl_{\RR}^{\ell^1}(X)$ as
\begin{equation}
\left\|\overline{D}\right\|_{\iota,\|\cdot\|}:=\left\|D\right\|+\left\|\overline{D}-\iota(D)\right\|_{\ell^1}
\end{equation}
for $\overline{D}\in\aCl_{\RR}^{\ell^1}(X)$, where we regard $\overline{D}-\iota(D)\in\aC(X)$.

\begin{proposition}
Let $\iota:\Cl_{\RR}(X)\to\aCl_{\RR}^{\ell^1}(X)$ be a section of $\zeta$ and let $\|\cdot\|$ be a norm on $\Cl_{\RR}(X)$.
\begin{enumerate}
\item $\left(\aCl_{\RR}^{\ell^1}(X),\|\cdot\|_{\iota,\|\cdot\|}\right)$ is a Banach space.
\item Let $\iota':\Cl_{\RR}(X)\to\aCl_{\RR}^{\ell^1}(X)$ be another section and let $\|\cdot\|'$ be another norm.
Then $\|\cdot\|_{\iota',\|\cdot\|'}$ is equivalent to $\|\cdot\|_{\iota,\|\cdot\|}$.
\end{enumerate}
\end{proposition}

\begin{proof}
(1): If $\left(\overline{D}_n\right)_{n\geq 1}$ is a Cauchy sequence in $\aCl_{\RR}^{\ell^1}(X)$, then $\left(\zeta(\overline{D}_n)\right)_{n\geq 1}$ is a Cauchy sequence in $\Cl_{\RR}(X)$, and converges to an $E\in\Cl_{\RR}(X)$.
Set $(0,\bm{f}_n):=\overline{D}_n-\iota(\zeta(\overline{D}_n))$.
The sequence $(\bm{f}_n)_{n\geq 1}$ is then a Cauchy sequence in $\aC(X)$, and converges to a $\bm{g}\in\aC(X)$.
The sequence $\left(\overline{D}_n\right)_{n\geq 1}$ then converges to $\iota(E)+(0,\bm{g})$.

(2): It suffices to show $\|\cdot\|_{\iota',\|\cdot\|'}\leq C\|\cdot\|_{\iota,\|\cdot\|}$ for a $C>0$.
We choose a basis $A_1,\dots,A_l$ for $\Cl_{\RR}(X)$ and set
\[
\|a_1A_1+\dots+a_lA_l\|_1:=|a_1|+\dots+|a_l|.
\]
We can find a constant $C_1\geq 1$ such that $\|\cdot\|'\leq C_1\|\cdot\|$ and such that $\|\cdot\|_1\leq C_1\|\cdot\|$.
We set $(0,\bm{f}_i):=\iota(A_i)-\iota'(A_i)$ for each $i$, and set
\[
C_2:=\max_{1\leq i\leq l}\left\{\left\|\iota(A_i)-\iota'(A_i)\right\|_{\ell^1},1\right\}.
\]
Then, for any $\overline{D}\in\aCl_{\RR}^{\ell^1}(X)$ with $D=a_1A_1+\dots+a_lA_l$,
\begin{align*}
\left\|\overline{D}\right\|_{\iota',\|\cdot\|'} & = \|D\|'+\left\|\overline{D}-\iota'(D)\right\|_{\ell^1} \\
&\leq C_1\|D\|+\|\overline{D}-\iota(D)\|_{\ell^1}+\sum_{i=1}^l|a_i|\|\iota(A_i)-\iota'(A_i)\|_{\ell^1} \\
&\leq C_1\|D\|+\|\overline{D}-\iota(D)\|_{\ell^1}+C_2\|\overline{D}\|_1\leq 2C_1C_2\|\overline{D}\|_{\iota,\|\cdot\|}.
\end{align*}
\end{proof}

\subsection{Arithmetic volume function}
\label{subsec:arithmetic:volume:function}

The following is a key idea to introduce the notion of $\ell^1$-adelic $\RR$-Cartier divisors.

\begin{lemma}
\label{lem:difference:adding:adelic:sequence}
Let $\mathscr{X}$ be a normal, projective, and geometrically connected arithmetic variety over $\Spec(O_K)$, and let $\overline{\mathscr{D}}\in\aDiv_{\RR}(\mathscr{X})$.
Suppose that every irreducible component of $\mathscr{D}$ is Cartier.
Let $U=U_{(\mathscr{X},\mathscr{D})}$ be a nonempty open subset of $\Spec(O_K)$ having the following properties.
\begin{enumerate}
\item[(a)] $\pi_U:\mathscr{X}_U\to U$ is geometrically reduced and geometrically irreducible.
\item[(b)] For every $v\in U$, $\ord_{\pi_U^{-1}(v)}(\mathscr{D})=0$.
\end{enumerate}
Then, for every $v\in U$ and $\phi\in H^0(D)\setminus\{0\}$, one has
\[
\inf_{x\in X_v^{\rm an}}\left\{g_v^{(\mathscr{X},\mathscr{D})}(x)-\log|\phi|^2(x)\right\}\in(2\log\sharp\widetilde{K}_v)\ZZ.
\]
\end{lemma}

\begin{proof}
By assumption, every irreducible component of $\mathscr{D}|_{\mathscr{X}_{K_v^{\circ}}}$ is Cartier, so we can write
\[
\mathscr{D}|_{\mathscr{X}_{K_v^{\circ}}}=a_1\mathscr{D}_1+\dots+a_r\mathscr{D}_r
\]
with $a_i\in\RR$ and prime Cartier divisors $\mathscr{D}_i$.

We choose a finite affine open covering $\left(\mathscr{U}_{\lambda}\right)_{\lambda}$ of $\mathscr{X}_{K_v^{\circ}}$ such that $\mathscr{U}_{\lambda}\cap\widetilde{\mathscr{X}}_v\neq\emptyset$ and $\mathscr{D}_i\cap\mathscr{U}_{\lambda}$ is principal with equation $f_{i,\lambda}$ for each $\lambda$.
We set $\mathscr{U}_{\lambda}=\Spec(\mathscr{A}_{\lambda})$ with finitely generated and integrally closed $K_v^{\circ}$-algebra $\mathscr{A}_{\lambda}$, and set $U_{\lambda}:=\Spec(\mathscr{A}_{\lambda}\otimes_{K_v^{\circ}}K)$.
We then have
\[
X_v^{\rm an}=\bigcup_{\lambda}\left(U_{\lambda}\right)_{v,\mathscr{U}_{\lambda}}^{\rm an}\quad\text{and}\quad
\psi_{\lambda}:=\phi\cdot f_{1,\lambda}^{\lfloor a_1\rfloor}\cdots f_{r,\lambda}^{\lfloor a_r\rfloor}\in\mathscr{A}_{\lambda}
\]
for every $\phi\in H^0(D)\setminus\{0\}$ and $\lambda$.

By Lemma~\ref{lem:maximum:modulus:principle}, the function
\[
\left(U_{\lambda}\right)_{v,\mathscr{U}_{\lambda}}^{\rm an}\to\RR,\quad x\mapsto |\psi_{\lambda}|(x)\cdot |f_{1,\lambda}|^{a_1-\lfloor a_1\rfloor}(x)\cdots |f_{r,\lambda}|^{a_r-\lfloor a_r\rfloor}(x),
\]
attains its maximum at the single point in $\Gamma(X_v^{\rm an})\cap \left(U_{\lambda}\right)_{v,\mathscr{U}_{\lambda}}^{\rm an}$ that corresponds to the fiber $\widetilde{\mathscr{X}}_v$.
Let $\varpi_v$ be a uniformizer of $K_v$.
Since
\[
\ord_{\widetilde{\mathscr{X}}_v}(\varpi_v)=1\quad\text{and}\quad\ord_{\widetilde{\mathscr{X}}_v}(f_{1,\lambda})=\cdots=\ord_{\widetilde{\mathscr{X}}_v}(f_{r,\lambda})=0,
\]
we have
\begin{align*}
\frac{1}{2}\inf_{x\in \left(U_{\lambda}\right)_{v,\mathscr{U}_{\lambda}}^{\rm an}}\left\{g_v^{(\mathscr{X},\mathscr{D})}(x)-\log|\phi|(x)\right\} & = \frac{\ord_{\widetilde{\mathscr{X}}_v}\left(\phi\cdot f_{1,\lambda}^{a_1}\cdots f_{r,\lambda}^{a_r}\right)}{\ord_{\widetilde{\mathscr{X}}_v}(\varpi_v)}\cdot\log\sharp\widetilde{K}_v \\
& = \ord_{\widetilde{\mathscr{X}}_v}(\phi)\log\sharp\widetilde{K}_v\in (\log\sharp\widetilde{K}_v)\ZZ
\end{align*}
for every $\lambda$.
We have thus proved the lemma.
\end{proof}

\begin{proposition}
\label{prop:difference:adding:continuous:function}
Let $X$ be a normal, projective, and geometrically connected $K$-variety and let $\mu:\widetilde{X}\to X$ be a resolution of singularities of $X$.
Let $\overline{D}\in\aDiv_{\RR}(X)$, and let $\bm{a}=\sum_{v\in M_K}a_v[v]\in C_{\rm tot}(\Spec(K))$ with $\bm{a}\geq 0$.
\begin{enumerate}
\item[(I)] Let $U$ be a nonempty open subset of $\Spec(O_K)$ over which a model of definition for $\overline{D}$ exists.
\suspend{enumerate}
We choose an $O_K$-model $(\widetilde{\mathscr{X}},\widetilde{\mathscr{D}})$ of $(\widetilde{X},\mu^*D)$ such that $(\widetilde{\mathscr{X}}_U,\widetilde{\mathscr{D}}_U)$ gives a $U$-model of definition for $\mu^*\overline{D}$ and such that every irreducible component of $\widetilde{\mathscr{D}}$ is Cartier.
\resume{enumerate}
\item[(II)] Let $U_{(\widetilde{\mathscr{X}},\widetilde{\mathscr{D}})}$ be a nonempty open subset of $U$ such that $\pi:\widetilde{\mathscr{X}}_{U_{(\widetilde{\mathscr{X}},\widetilde{\mathscr{D}})}}\to U_{(\widetilde{\mathscr{X}},\widetilde{\mathscr{D}})}$ is smooth and such that $\ord_{\pi^{-1}(v)}(\widetilde{\mathscr{D}})=0$ for every $v\in U_{(\widetilde{\mathscr{X}},\widetilde{\mathscr{D}})}$.
\item[(III)] $U_{\bm{a}}:=\left\{v\in M_K^{\rm fin}\,:\,a_v<2\log\sharp\widetilde{K}_v\right\}$.
\end{enumerate}
We set
\[
\bm{a}':=\sum_{v\notin U_{(\widetilde{\mathscr{X}},\widetilde{\mathscr{D}})}\cap U_{\bm{a}}}a_v[v].
\]
Then the following holds.
\begin{enumerate}
\item If a $\bm{b}\in C_{\rm tot}(\Spec(K))$ satisfies $\bm{b}\geq\bm{a}$, then $U_{\bm{b}}\subset U_{\bm{a}}$.
\item For any $\mathcal{V}\in\BC_{\RR}(X)$, one has
\[
\aHzf\left(\overline{D}+(0,\bm{a});\mathcal{V}\right)=\aHzf\left(\overline{D}+(0,\bm{a}');\mathcal{V}\right).
\]
\item If $\sharp(M_K^{\rm fin}\setminus U_{\bm{a}})$ is finite (in particular, if $\bm{a}$ is a bounded sequence), then $\ah\left(\overline{D}+(0,\bm{a});\mathcal{V}\right)$ is finite for every $\mathcal{V}\in\BC_{\RR}(X)$ and $\ast=\text{s, ss}$.
\end{enumerate}
\end{proposition}

\begin{proof}
The assertion (1) is obvious.

(2): Since $\bm{a}\geq 0$, the inclusion $\supset$ is obvious.
Suppose $v\in U_{(\widetilde{\mathscr{X}},\widetilde{\mathscr{D}})}\cap U_{\bm{a}}$; hence, in particular,
\begin{equation}
\label{eqn:adding:continuous:functions:1}
0\geq -a_v>-2\log\sharp\widetilde{K}_v.
\end{equation}
If $\phi\in H^0(\mu^*D;\mathcal{V})\setminus\{0\}=H^0(D;\mathcal{V})\setminus\{0\}$ satisfies
\[
g_v^{\overline{D}}(x)+a_v-\log|\phi|^2(x)\geq 0
\]
for every $x\in X_v^{\rm an}$, then
\[
g_v^{(\widetilde{\mathscr{X}},\widetilde{\mathscr{D}})}(x')-\log|\phi|^2(x')\geq -a_v
\]
for every $x'\in\widetilde{X}_v^{\rm an}$.
By Lemma~\ref{lem:difference:adding:adelic:sequence} and \eqref{eqn:adding:continuous:functions:1}, we have
\[
\inf_{x'\in\widetilde{X}_v^{\rm an}}\left\{g_v^{(\widetilde{\mathscr{X}},\widetilde{\mathscr{D}})}(x')-\log|\phi|^2(x')\right\}=\inf_{x\in X_v^{\rm an}}\left\{g_v^{\overline{D}}(x)-\log|\phi|^2(x)\right\}\geq 0.
\]
Hence $\phi\in\aHzf\left(\overline{D}+(0,\bm{a});\mathcal{V}\right)$ implies $\phi\in\aHzf\left(\overline{D}+(0,\bm{a}');\mathcal{V}\right)$.

If $M_K\setminus U_{\bm{a}}$ is finite, then so is $M_K\setminus (U_{(\widetilde{\mathscr{X}},\widetilde{\mathscr{D}})}\cap U_{\bm{a}})$.
Hence, the assertion (2) implies the assertion (3) (see \cite[Proposition~4.3.1(3)]{MoriwakiAdelic}).
\end{proof}

\begin{proposition}
\label{prop:basic:ah:of:ell1:adelic:divisors}
Let $X$ be a normal, projective, and geometrically connected $K$-variety and let $\ast=\text{s or ss}$.
\begin{enumerate}
\item To each $\overline{D}\in\aDiv_{\RR}(X)$, one can assign a constant $\delta(\overline{D})>0$, which depends only on $\overline{D}$ and $X$, such that
\[
0\leq\ah\left(\overline{D}+(0,\bm{f});\mathcal{V}\right)-\ah\left(\overline{D};\mathcal{V}\right)\leq\left(\frac{3}{2}\left\|\bm{f}\right\|_{\ell^1}+\delta(\overline{D})\right)\dim_{\QQ}H^0(D;\mathcal{V}).
\]
for every $\bm{f}\in\aC(X)$ and $\mathcal{V}\in\BC_{\RR}(X)$.
Moreover, one can assume that
\[
\delta(t\overline{D})=\delta(\overline{D})
\]
holds for every $t\in \RR\setminus\{0\}$.
\item For any $(\overline{D};\mathcal{V})\in\aBDiv_{\RR,\RR}^{\ell^1}(X)$, $\aHzs{\ast}\left(\overline{D};\mathcal{V}\right)$ is a finite set.
\item For any $(\overline{D};\mathcal{V})\in\aBDiv_{\RR,\RR}^{\ell^1}(X)$, $\left(H^0(D;\mathcal{V}),(\|\cdot\|_{v,\sup}^{\overline{D}})_{v\in M_K}\right)$ is an adelically normed $K$-vector space.
\end{enumerate}
\end{proposition}

\begin{proof}
(1): Set
\begin{equation}
\bm{a}:=\sum_{v\in M_K}\|f_v\|_{\sup}[v]\in\aC(\Spec(K)).
\end{equation}
For each $v\in M_K^{\rm fin}$, we denote by $p_v$ the prime number satisfying $p_v\ZZ=\mathfrak{p}_v\cap\ZZ$.

Let $\mu:\widetilde{X}\to X$ be a resolution of singularities of $X$ and let $U$ be a nonempty open subset of $\Spec(O_K)$ over which a model of definition for $\overline{D}$ exists.
Let $(\widetilde{\mathscr{X}},\widetilde{\mathscr{D}})$ be an $O_K$-model of $(\widetilde{X},\mu^*D)$ such that $(\widetilde{\mathscr{X}}_U,\widetilde{\mathscr{D}}_U)$ gives a $U$-model of definition for $\mu^*\overline{D}$ and such that every irreducible component of $\widetilde{\mathscr{D}}$ is Cartier.

We choose the two nonempty open subsets $U_{(\widetilde{\mathscr{X}},\widetilde{\mathscr{D}})}$ and $U_{\bm{a}}$ as in Proposition~\ref{prop:difference:adding:continuous:function}; namely,
\begin{itemize}
\item $U_{(\widetilde{\mathscr{X}},\widetilde{\mathscr{D}})}$ is chosen to satisfy that $U_{(\widetilde{\mathscr{X}},\widetilde{\mathscr{D}})}\subset U$, that $\pi:\widetilde{\mathscr{X}}_{U_{(\widetilde{\mathscr{X}},\widetilde{\mathscr{D}})}}\to U_{(\widetilde{\mathscr{X}},\widetilde{\mathscr{D}})}$ is smooth, and that $\ord_{\pi^{-1}(v)}(\widetilde{\mathscr{D}})=0$ for every $v\in U_{(\widetilde{\mathscr{X}},\widetilde{\mathscr{D}})}$, and
\item $U_{\bm{a}}:=\left\{v\in M_K^{\rm fin}\,:\,a_v<2\log\sharp\widetilde{K}_v\right\}$.
\end{itemize}
We divide $M_K^{\rm fin}$ into three disjoint subsets: $S_1:=U_{(\widetilde{\mathscr{X}},\widetilde{\mathscr{D}})}\cap U_{\bm{a}}$,
\[
S_2:=\left\{v\in M_K^{\rm fin}\,:\,\text{$2\log\sharp\widetilde{K}_v\leq a_v$ and $2\leq\log(p_v)$}\right\},
\]
and $S_3:=M_K^{\rm fin}\setminus (S_1\cup S_2)$.
Note that only $S_1$ is an infinite subset and $S_3$ is contained in a finite subset
\begin{equation}
S_3':=\left(M_K\setminus U_{(\widetilde{\mathscr{X}},\widetilde{\mathscr{D}})}\right)\cup\{v\in M_K^{\rm fin}\,:\,\log(p_v)<2\},
\end{equation}
which is determined only by $U_{(\widetilde{\mathscr{X}},\widetilde{\mathscr{D}})}$.
Put
\begin{equation}
\bm{a}':=\sum_{v\in S_2}a_v[v]+\sum_{v\in S_3}a_v[v].
\end{equation}
By Proposition~\ref{prop:difference:adding:continuous:function}(2), Lemma~\ref{lem:adelic:rescaling:adelically:normed:vector:spaces}, and \eqref{eqn:elementary:property:of:ah:rescaling}, we have
\begin{align}
& \ah\left(\overline{D}+(0,\bm{f});\mathcal{V}\right)-\ah\left(\overline{D};\mathcal{V}\right) \label{eqn:basic:ah:of:ell1:adelic:divisors:1}\\
& \quad \leq \ah\left(\overline{D}+(0,\bm{a});\mathcal{V}\right)-\ah\left(\overline{D};\mathcal{V}\right) \nonumber\\
& \quad =\ah\left(\overline{D}+(0,\bm{a}');\mathcal{V}\right)-\ah\left(\overline{D};\mathcal{V}\right) \nonumber\\
& \quad \leq\left(\sum_{v\in S_2\cup S_3}\left(\left\lceil\frac{\|f_v\|_{\sup}}{-2\log|p_v|_v}\right\rceil\log(p_v)+2\right)+\frac{\|f_{\infty}\|_{\sup}}{2}+2\right)\dim_{\QQ}H^0(D;\mathcal{V}). \nonumber
\end{align}
We can estimate the sum with respect to $S_2$ as
\begin{align}
&\sum_{v\in S_2}\left(\left\lceil\frac{\|f_v\|_{\sup}}{-2\log|p_v|_v}\right\rceil\log(p_v)+2\right) \label{eqn:basic:ah:of:ell1:adelic:divisors:2}\\
&\qquad \leq\sum_{v\in S_2}\left(\frac{\|f_v\|_{\sup}}{2\ord_v(p_v)[\widetilde{K}_v:\FF_{p_v}]}+2\log(p_v)\right) \nonumber\\
&\qquad \leq\sum_{v\in S_2}\left(\frac{\|f_v\|_{\sup}}{2\ord_v(p_v)[\widetilde{K}_v:\FF_{p_v}]}+\frac{\|f_v\|_{\sup}}{[\widetilde{K}_v:\FF_{p_v}]}\right) \leq\frac{3}{2}\sum_{v\in S_2}\|f_v\|_{\sup} \nonumber
\end{align}
and the sum with respect to $S_3$ as
\begin{equation}
\label{eqn:basic:ah:of:ell1:adelic:divisors:4}
\sum_{v\in S_3}\left(\left\lceil\frac{\|f_v\|_{\sup}}{-2\log|p_v|_v}\right\rceil\log(p_v)+2\right)\leq\sum_{v\in S_3}\left(\frac{1}{2}\|f_v\|_{\sup}+\log(p_v)+2\right).
\end{equation}
Hence, if we set $p_{\infty}:=1$ and
\begin{equation}
\label{eqn:basic:ah:of:ell1:adelic:divisors:3}
\delta(\overline{D}):=\sum_{v\in S_3'}(\log(p_v)+2),
\end{equation}
then we obtain
\[
\ah\left(\overline{D}+(0,\bm{f});\mathcal{V}\right)-\ah\left(\overline{D};\mathcal{V}\right)\leq\left(\frac{3}{2}\|\bm{f}\|_{\ell^1}+\delta(\overline{D})\right)\dim_{\QQ}H^0(D;\mathcal{V})
\]
by \eqref{eqn:basic:ah:of:ell1:adelic:divisors:1}, \eqref{eqn:basic:ah:of:ell1:adelic:divisors:2}, and \eqref{eqn:basic:ah:of:ell1:adelic:divisors:4}.
Since the constant $\delta(\overline{D})$ depends only on $U_{(\widetilde{\mathscr{X}},\widetilde{\mathscr{D}})}$, we have $\delta(t\overline{D})=\delta(\overline{D})$ for every $t\in\RR\setminus\{0\}$.

The assertion (2) is obvious from the assertion (1).
The assertion (3) follows from the assertion (2) and the fact that $\aHzf(\overline{D};\mathcal{V})$ contains $H^0(\mathscr{D};\mathcal{V})$ for any $(\mathscr{X},\overline{\mathscr{D}})\in\Mod(\overline{D})$.
\end{proof}

\begin{proposition}
\label{prop:finiteness:of:arithmetic:volumes}
Let $\ast=\text{s or ss}$.
For any $(\overline{D};\mathcal{V})\in\aBDiv_{\RR,\RR}^{\ell^1}(X)$,
\[
\limsup_{\substack{m\in\ZZ, \\ m\to+\infty}}\frac{\ah\left(m\overline{D};m\mathcal{V}\right)}{m^{\dim X+1}/(\dim X+1)!}
\]
is finite.
\end{proposition}

\begin{proof}
Take a $\overline{D}_0\in\aDiv_{\RR}(X)$ such that $\zeta(\overline{D}_0)=\zeta(\overline{D})$ and $\overline{D}_0\leq\overline{D}$.
By Proposition~\ref{prop:basic:ah:of:ell1:adelic:divisors}(1), we have
\begin{align*}
&\limsup_{\substack{m\in\ZZ, \\ m\to+\infty}}\frac{\ah\left(m\overline{D};m\mathcal{V}\right)}{m^{\dim X+1}/(\dim X+1)!} \\
&\qquad \leq\limsup_{\substack{m\in\ZZ, \\ m\to+\infty}}\frac{\ah\left(m\overline{D}_0;m\mathcal{V}\right)}{m^{\dim X+1}/(\dim X+1)!} \\
&\qquad\qquad +(\dim X+1)\limsup_{\substack{m\in\ZZ, \\ m\to+\infty}}\left(\frac{3}{2}\left\|\overline{D}-\overline{D}_0\right\|_{\ell^1}+\frac{\delta(m\overline{D}_0)}{m}\right)\frac{\dim_{\QQ}(mD;m\mathcal{V})}{m^{\dim X}/(\dim X)!} \\
&\qquad \leq\avol(\overline{D}_0;\mathcal{V})+\frac{3}{2}(\dim X+1)[K:\QQ]\left\|\overline{D}-\overline{D}_0\right\|_{\ell^1}\vol(D;\mathcal{V})<+\infty.
\end{align*}
\end{proof}

\begin{definition}
Given a $(\overline{D};\mathcal{V})\in\aBDiv_{\RR,\RR}^{\ell^1}(X)$, we define
\begin{equation}
\label{defn:arithmetic:volume:of:ell1pairs}
\avol(\overline{D};\mathcal{V}):=\limsup_{\substack{m\in\ZZ, \\ m\to+\infty}}\frac{\als\left(m\overline{D};m\mathcal{V}\right)}{m^{\dim X+1}/(\dim X+1)!}.
\end{equation}
By Proposition~\ref{prop:finiteness:of:arithmetic:volumes}, $\avol(\overline{D};\mathcal{V})$ is finite and
\begin{equation}
\label{eqn:difference:adding:cont:fcn}
0\leq\avol(\overline{D};\mathcal{V})-\avol(\overline{D}_0;\mathcal{V})\leq\frac{3}{2}(\dim X+1)[K:\QQ]\vol(D;\mathcal{V})\cdot\left\|\overline{D}-\overline{D}_0\right\|_{\ell^1}
\end{equation}
for every $\overline{D}_0\in\aDiv_{\RR}(X)$ with $\zeta(\overline{D}_0)=\zeta(\overline{D})$ and $\overline{D}_0\leq\overline{D}$.
Moreover, we can easily observe
\begin{equation}
\avol(\overline{D};\mathcal{V})=\limsup_{\substack{m\in\ZZ, \\ m\to+\infty}}\frac{\alss\left(m\overline{D};m\mathcal{V}\right)}{m^{\dim X+1}/(\dim X+1)!}.
\end{equation}
\end{definition}

\begin{proposition}
\label{prop:definition:of:height}
Let $X$ be a normal, projective, and geometrically connected $K$-variety, let $\overline{D}=\left(D,\sum_{v\in M_K}g_v^{\overline{D}}[v]\right)\in\aDiv_{\RR}^{\ell^1}(X)$, and let $x\in X(\overline{K})$.
The infinite sum
\[
\Delta:=\sum_{v\in M_K^{\rm fin}}\sum_{\substack{w\in M_{\kappa(x)}^{\rm fin}, \\ w|v}}[\kappa(x)_w:K_v]g_v^{\overline{D}}(x^w)+\sum_{\sigma:\kappa(x)\to\CC}g_{\infty}^{\overline{D}}(x^{\sigma})
\]
then converges, where the limit is taken with respect to the net indexed by all the finite subsets of $M_K^{\rm fin}$, $x^w\in X_v^{\rm an}$ is a point corresponding to $(\kappa(x),|\cdot|_w)$, and $x^{\sigma}\in X_{\infty}^{\rm an}$ is a point defined as $\Spec(\CC)\xrightarrow{\sigma}\Spec(\kappa(x))\xrightarrow{x}X$.
\end{proposition}

\begin{proof}
Let $(\mathscr{X},\overline{\mathscr{D}})\in\Mod(\overline{D})$ and let $(0,\bm{f}):=\overline{D}-\overline{\mathscr{D}}^{\rm ad}$.
We write
\[
\mathscr{D}=a_1\mathscr{D}_1+\dots+a_r\mathscr{D}_r
\]
with $a_i\in\RR$ and effective Cartier divisors $\mathscr{D}_i$.
Then
\begin{align*}
&\Delta-\sum_{\sigma:\kappa(x)\to\CC}g_{\infty}^{\overline{D}}(x^{\sigma}) \\
&\quad =\sum_{v\in M_K^{\rm fin}}\sum_{\substack{w\in M_{\kappa(x)}^{\rm fin}, \\ w|v}}[\kappa(x)_w:K_v]g_v^{(\mathscr{X},\mathscr{D})}(x^w)+\sum_{v\in M_K^{\rm fin}}\sum_{\substack{w\in M_{\kappa(x)}^{\rm fin}, \\ w|v}}[\kappa(x)_w:K_v]f_v(x^w) \\
&\quad =2\sum_{i=1}^ra_i\log\sharp(O_{\kappa(x)}(\mathscr{D}_i)/O_{\kappa(x)})+\sum_{v\in M_K^{\rm fin}}\sum_{\substack{w\in M_{\kappa(x)}^{\rm fin}, \\ w|v}}[\kappa(x)_w:K_v]f_v(x^w)
\end{align*}
(see \cite[section~2.3]{MoriwakiAdelic}).
Let $\varepsilon>0$.
Since $\bm{f}\in\aC(X)$, one can find a finite subset $S_0\subset M_K^{\rm fin}$ such that
\begin{multline*}
\left|\sum_{v\in S_1}\sum_{\substack{w\in M_{\kappa(x)}, \\ w|v}}[\kappa(x)_w:K_v]f_v(x^w)-\sum_{v\in S_2}\sum_{\substack{w\in M_{\kappa(x)}, \\ w|v}}[\kappa(x)_w:K_v]f_v(x^w)\right|_{\infty} \\
\leq [\kappa(x):K]\sum_{v\in M_K^{\rm fin}\setminus S_0}\|f_v\|_{\sup}\leq \varepsilon
\end{multline*}
for every finite subsets $S_1,S_2$ of $M_K^{\rm fin}$ such that $S_1\supset S_0$ and $S_2\supset S_0$.
So, by completeness of $\RR$, $\Delta$ converges.
\end{proof}

\begin{definition}
An $\ell^1$-adelic $\RR$-Cartier divisor $\overline{D}$ on $X$ determines a height function $h_{\overline{D}}:X(\overline{K})\to\RR$ by
\[
h_{\overline{D}}(x):=\frac{1}{[\kappa(x):\QQ]}\left(\frac{1}{2}\sum_{v\in M_K^{\rm fin}}\sum_{\substack{w\in M_{\kappa(x)}^{\rm fin}, \\ w|v}}[\kappa(x)_w:K_v]g_v^{\overline{D}}(x^w)+\frac{1}{2}\sum_{\sigma:\kappa(x)\to\CC}g_{\infty}^{\overline{D}}(x^{\sigma})\right),
\]
which is well-defined by Proposition~\ref{prop:definition:of:height} above, and belongs, up to $O(1)$, to the Weil height function corresponding to $D$.
Moreover, from the proof of Proposition~\ref{prop:definition:of:height}, one deduces
\begin{equation}
\sup_{x\in X(\overline{K})}\left|h_{\overline{D}}(x)-h_{\overline{D}'}(x)\right|\leq \frac{1}{2}\left\|\overline{D}-\overline{D}'\right\|_{\ell^1}
\end{equation}
for every $\overline{D},\overline{D}'\in\aDiv_{\RR}^{\ell^1}(X)$ with $\zeta(\overline{D})=\zeta(\overline{D}')$.

We abbreviate
\[
e_{\max}\left(\overline{D};\mathcal{V}\right):=e_{\max}\left(H^0(D;\mathcal{V}),(\|\cdot\|_{v,\sup}^{\overline{D}})_{v\in M_K}\right)
\]
(see \eqref{eqn:definition:of:e:max}), and define the \emph{essential minimum} of $\overline{D}$ as
\begin{equation}
\essmin_{x\in X(\overline{K})}h_{\overline{D}}(x)=\sup_{Y\subsetneq X}\inf_{x\in (X\setminus Y)(\overline{K})}h_{\overline{D}}(x),
\end{equation}
where the supremum is taken over all the closed proper subvarieties of $X$.
\end{definition}

\begin{lemma}
For any $\overline{D}\in\aDiv_{\RR}^{\ell^1}(X)$, we have
\[
\lim_{\substack{m\in\ZZ, \\ m\to+\infty}}\frac{e_{\max}\left(m\overline{D}\right)}{m}\leq\essmin_{x\in X(\overline{K})}h_{\overline{D}}(x)<+\infty.
\]
\end{lemma}

\begin{proof}
Note that $e_{\max}(\overline{D})=\min\left\{\lambda\in\RR\,:\,\aHzsm(\overline{D}+(0,2\lambda[\infty]))\neq\{0\}\right\}$ and
\[
\lim_{\substack{m\in\ZZ, \\ m\to+\infty}}\frac{e_{\max}(m\overline{D})}{m}=\sup_{m\in\ZZ_{>0}}\frac{e_{\max}(m\overline{D})}{m}
\]
by Fekete's lemma.
Let $\lambda\in\RR_{\geq 0}$, let $\phi\in\aHzsm(m\overline{D}+(0,2\lambda[\infty]))\setminus\{0\}$, and let $Z:=\Supp(mD+(\phi))$.
For every $x\in (X\setminus Z)(\overline{K})$, we have
\[
h_{\overline{D}}(x)\geq\frac{1}{2m}\inf_{x\in (X\setminus Z)_{\infty}^{\rm an}}g_{\infty}^{m\overline{D}}(x)\geq\frac{\lambda}{m}.
\]
Hence we have the first inequality.

To show the second inequality, we write
\[
\overline{D}=a_1\overline{D}_1+\dots+a_r\overline{D}_r+(0,\bm{f})
\]
such that $a_i\in\RR$, $\overline{D}_i\in\aDiv(X)$, $\bm{f}\in\aC(X)$, and $\overline{D}_i$ are all effective (see \cite[Proposition~2.4.2(1)]{MoriwakiZar}).
We set
\[
\overline{D}':=\lceil a_1\rceil\overline{D}_1+\dots+\lceil a_r\rceil\overline{D}_r\quad\text{and}\quad\Sigma:=\bigcup_{i=1}^r\Supp(D_i).
\]
By \cite[Proposition~2.6]{Boucksom_Chen}, $\left\{x\in (X\setminus\Sigma)(\overline{K})\,:\, h_{\overline{D}'}(x)\leq C\right\}$ is Zariski dense in $X$ for a constant $C$.

If $x\in (X\setminus\Sigma)(\overline{K})$, then $h_{\overline{D}}(x)\leq h_{\overline{D}'}(x)+\|\bm{f}\|_{\ell^1}$.
Hence
\[
\left\{x\in X(\overline{K})\,:\,h_{\overline{D}}(x)\leq C+\|\bm{f}\|_{\ell^1}\right\}\supset\left\{x\in (X\setminus\Sigma)(\overline{K})\,:\, h_{\overline{D}'}(x)\leq C\right\},
\]
and the left-hand side is also Zariski dense in $X$.
It implies that the essential minimum is bounded from above by $C+\|\bm{f}\|_{\ell^1}$.
\end{proof}

\begin{lemma}
\label{lem:the:case:volume:is:zero}
For any $(\overline{D};\mathcal{V})\in\aBDiv_{\RR,\RR}(X)$, one has
\[
0\leq\avol(\overline{D};\mathcal{V})\leq (\dim X+1)[K:\QQ]\vol(D;\mathcal{V})\max\left\{\lim_{\substack{m\in\ZZ, \\ m\to+\infty}}\frac{e_{\max}(m\overline{D};m\mathcal{V})}{m},0\right\}.
\]
\end{lemma}

\begin{proof}
By Gillet--Soul\'e's formula \cite[Proposition~6]{Gillet_Soule91}, we have
\begin{align*}
0\leq \als\left(m\overline{D};m\mathcal{V}\right) &\leq \max\left\{e_{\max}(m\overline{D};m\mathcal{V}),0\right\}\cdot\rk H^0(mD;m\mathcal{V}) \\
& \qquad\qquad\qquad\qquad +2\left(\rk H^0(mD)+\log(\rk H^0(mD))!\right)
\end{align*}
for every $m\in\ZZ_{>0}$.
Therefore,
\begin{align*}
0\leq\avol(\overline{D};\mathcal{V}) &\leq (\dim X+1)[K:\QQ]\max\left\{\lim_{\substack{m\in\ZZ, \\ m\to+\infty}}\frac{e_{\max}(m\overline{D};m\mathcal{V})}{m},0\right\} \\
&\qquad\qquad\qquad\qquad\qquad\qquad\qquad \cdot\limsup_{\substack{m\in\ZZ, \\ m\to+\infty}}\frac{\dim_KH^0(mD;m\mathcal{V})}{m^{\dim X}/(\dim X)!} \\
&=(\dim X+1)[K:\QQ]\vol(D;\mathcal{V})\max\left\{\lim_{\substack{m\in\ZZ, \\ m\to+\infty}}\frac{e_{\max}(m\overline{D};m\mathcal{V})}{m},0\right\}.
\end{align*}
\end{proof}

\begin{lemma}
\label{lem:zeta:value:fixed}
Let $(\overline{D};\mathcal{V})\in\aBDiv_{\RR,\RR}^{\ell^1}(X)$.
Let $\left(\overline{D}_n\right)_{n\geq 1}$ be an increasing sequence in $\aDiv_{\RR}^{\ell^1}(X)$ such that $\zeta(\overline{D}_n)=D$ and such that $\left\|\overline{D}-\overline{D}_n\right\|_{\ell^1}\to 0$ as $n\to+\infty$.
One then has
\[
\avol\left(\overline{D};\mathcal{V}\right)=\lim_{n\to+\infty}\avol\left(\overline{D}_n;\mathcal{V}\right).
\]
\end{lemma}

\begin{proof}
Since $\left(\overline{D}_n\right)_{n\geq 1}$ is an increasing sequence, we can assume $\overline{D}_n\in\aDiv_{\RR}(X)$ for every $n\geq 1$ by Proposition~\ref{prop:definition:of:ell1:adelic}.
Hence, by \eqref{eqn:difference:adding:cont:fcn},
\[
\left|\avol(\overline{D};\mathcal{V})-\avol(\overline{D}_n;\mathcal{V})\right|\leq\frac{3}{2}(\dim X+1)[K:\QQ]\vol(D;\mathcal{V})\cdot\left\|\overline{D}-\overline{D}_n\right\|_{\ell^1}\to 0
\]
as $n\to+\infty$.
\end{proof}

\begin{proposition}
\label{prop:difference:adding:cont:fcn:final:version}
Let $(\overline{D};\mathcal{V})\in\aBDiv_{\RR,\RR}^{\ell^1}(X)$.
For any $\bm{f}\in\aC(X)$, we have
\[
\left|\avol\left(\overline{D}+(0,\bm{f});\mathcal{V}\right)-\avol\left(\overline{D};\mathcal{V}\right)\right|\leq\frac{1}{2}(\dim X+1)[K:\QQ]\vol(D;\mathcal{V})\cdot\|\bm{f}\|_{\ell^1}.
\]
\end{proposition}

\begin{proof}
Let $\left(\overline{D}_n\right)_{n\geq 1}$ be an increasing sequence in $\aDiv_{\RR}(X)$ such that $\zeta(\overline{D}_n)=D$ and such that $\left\|\overline{D}-\overline{D}_n\right\|_{\ell^1}\to 0$ as $n\to+\infty$, and let $\left(\bm{f}_n\right)_{n\geq 1}$ be an increasing sequence in $C(X)$ such that $\|\bm{f}-\bm{f}_n\|_{\ell^1}\to 0$ as $n\to+\infty$.
By the same arguments as in \cite[Proposition~5.1.3]{MoriwakiAdelic}, Lemma~\ref{lem:adelic:rescaling:adelically:normed:vector:spaces} implies
\[
\left|\avol\left(\overline{D}_n+(0,\bm{f}_n);\mathcal{V}\right)-\avol\left(\overline{D}_n;\mathcal{V}\right)\right|\leq\frac{1}{2}(\dim X+1)[K:\QQ]\vol(D;\mathcal{V})\cdot\|\bm{f}_n\|_{\ell^1}.
\]
By taking $n\to+\infty$, we have the required assertion by Lemma~\ref{lem:zeta:value:fixed}.
\end{proof}

\subsection{Continuity of the arithmetic volume function}\label{subsec:ContHomog}

The purpose of this section is to establish the global continuity of the arithmetic volume function over $\aBDiv_{\RR,\RR}^{\ell^1}(X)$ along the directions of $\ell^1$-adelic $\RR$-Cartier divisors (see Theorem~\ref{thm:Continuity_Est3}).
To begin with, we show the homogeneity of the arithmetic volume function in the following form.

\begin{lemma}
\label{prop:positive:homogeneity}
Let $\mathscr{X}$ be a projective arithmetic variety of dimension $d+1$ having smooth generic fiber $\mathscr{X}_{\QQ}$.
Let $\overline{\mathscr{D}}\in\aDiv_{\QQ}(\mathscr{X};C^{\infty})$ and let $\mathcal{V}\in\BC_{\RR}(\mathscr{X})$ with $\mathcal{V}\geq 0$.
For any $p\in\ZZ_{>0}$, one has
\[
\avol\left(p\overline{\mathscr{D}};p\mathcal{V}\right)=p^{\dim X+1}\avol\left(\overline{\mathscr{D}};\mathcal{V}\right).
\]
\end{lemma}

\begin{proof}
First, we note the following.

\begin{claim}
\label{clm:proof:of:positive:homogeneity}
It suffices to show that, to each $\overline{\mathscr{D}}\in\aDiv_{\QQ}(\mathscr{X};C^{\infty})$, one can assign a $q_{\overline{\mathscr{D}}}\in\ZZ_{>0}$ such that the equality is true for all multiples of $q_{\overline{\mathscr{D}}}$.
\end{claim}

\begin{proof}[Proof of Claim~\ref{clm:proof:of:positive:homogeneity}]
For any $p\in\ZZ_{>0}$, one has
\begin{align*}
\avol\left(p\overline{\mathscr{D}};p\mathcal{V}\right) &=\frac{1}{(q_{\overline{\mathscr{D}}}q_{p\overline{\mathscr{D}}})^{\dim X+1}}\avol\left((pq_{\overline{\mathscr{D}}}q_{p\overline{\mathscr{D}}})\overline{D};(pq_{\overline{\mathscr{D}}}q_{p\overline{\mathscr{D}}})\mathcal{V}\right) \\
&=p^{\dim X+1}\avol\left(\overline{D};\mathcal{V}\right).
\end{align*}
\end{proof}

By Claim~\ref{clm:proof:of:positive:homogeneity}, it suffices to show the equality for every $p\in\ZZ_{>0}$ with
\[
\overline{\mathscr{D}}':=p\overline{\mathscr{D}}\in\aDiv(\mathscr{X};C^{\infty}).
\]
We fix an $\overline{\mathscr{E}}\in\aDiv(\mathscr{X})$ such that $\overline{\mathscr{E}}\geq 0$ and $\overline{\mathscr{E}}\pm\overline{\mathscr{D}}'\geq 0$.
By Theorem~\ref{thm:Continuity_Est2}, there is a constant $C>0$ such that
\[
0\leq\als\left(\OO_{\mathscr{X}}(m\overline{\mathscr{D}}'+\overline{\mathscr{E}});n\mathcal{V}\right)-\als\left(\OO_{\mathscr{X}}(m\overline{\mathscr{D}}'-\overline{\mathscr{E}});n\mathcal{V}\right)\leq Cm^{\dim X}(1+\log(m))
\]
for $m\in\ZZ_{>0}$ and $n\in\ZZ_{\geq 0}$.
Hence, for each $r=1,2,\dots,p-1$, we obtain
\begin{align*}
\limsup_{\substack{m\in\ZZ, \\ m\to+\infty}}\frac{\als\left((pm+r)\overline{\mathscr{D}};(pm+r)\mathcal{V}\right)}{(pm+r)^{\dim X+1}/(\dim X+1)!} &\leq\limsup_{\substack{m\in\ZZ, \\ m\to+\infty}}\frac{\als\left(\left(m+\frac{r}{p}\right)\overline{\mathscr{D}}';pm\mathcal{V}\right)}{(pm)^{\dim X+1}/(\dim X+1)!} \\
&= \limsup_{\substack{m\in\ZZ, \\ m\to+\infty}}\frac{\als\left(\OO_{\mathscr{X}}(m\overline{\mathscr{D}}');pm\mathcal{V}\right)}{(pm)^{\dim X+1}/(\dim X+1)!}.
\end{align*}
Therefore,
\begin{align*}
\avol\left(\overline{\mathscr{D}};\mathcal{V}\right) &=\max_{0\leq r<p}\left\{\limsup_{\substack{m\in\ZZ, \\ m\to+\infty}}\frac{\als\left((pm+r)\overline{\mathscr{D}};(pm+r)\mathcal{V}\right)}{(pm+r)^{\dim X+1}/(\dim X+1)!}\right\} \\
&=\limsup_{\substack{m\in\ZZ, \\ m\to+\infty}}\frac{\als\left(pm\overline{\mathscr{D}};pm\mathcal{V}\right)}{(pm)^{\dim X+1}/(\dim X+1)!}=\frac{1}{p^{\dim X+1}}\avol\left(p\overline{\mathscr{D}};p\mathcal{V}\right).
\end{align*}
\end{proof}

\begin{theorem}\label{thm:Continuity_Est3}
Let $X$ be a normal, projective, and geometrically connected $K$-variety.
Let $V$ be a finite-dimensional $\RR$-subspace of $\aDiv_{\RR}^{\ell^1}(X)$ endowed with a norm $\|\cdot\|_V$, let $\Sigma$ be a finite set of points on $X$, and let $B\in\RR_{>0}$.
For any $\varepsilon\in\RR_{>0}$, there exists a $\delta\in\RR_{>0}$ such that
\[
 \left|\avol\left(\overline{D}+(0,\bm{f});\mathcal{V}\right)-\avol\left(\overline{E};\mathcal{V}\right)\right|\leq\varepsilon
\]
for every $\overline{D},\overline{E}\in V$ with $\max\left\{\left\|\overline{D}\right\|_V,\left\|\overline{E}\right\|_V\right\}\leq B$ and $\left\|\overline{D}-\overline{E}\right\|_V\leq\delta$, $\bm{f}\in \aC(X)$ with $\|\bm{f}\|_{\ell^1}\leq\delta$, and $\mathcal{V}\in\BC_{\RR}(X)$ with $\{c_X(\nu)\,:\,\nu(\mathcal{V})>0\}\subset\Sigma$.
\end{theorem}

We need the following.

\begin{proposition}
\label{prop:Continuity_Est4}
Let $\mathscr{X}$ be a projective arithmetic variety of dimension $d+1$ such that $\mathscr{X}_{\QQ}$ is smooth.
Let $\overline{V}=(V,\|\cdot\|_V)$ be a couple of a finite-dimensional $\RR$-subspace $V$ of $\aDiv_{\RR}(\mathscr{X};C^{\infty})$ and a norm $\|\cdot\|_{V}$ on $V$, and let $\Sigma$ be a finite set of points on $\mathscr{X}$.
There then exists a positive constant $C_{\overline{V},\Sigma}>0$ such that
\[
\left|\avol\left(\overline{\mathscr{D}};\mathcal{V}\right)-\avol\left(\overline{\mathscr{D}}';\mathcal{V}\right)\right|\leq C_{\overline{V},\Sigma}\max\left\{\left\|\overline{\mathscr{D}}\right\|_V^d,\left\|\overline{\mathscr{D}}'\right\|_V^d\right\}\cdot\left\|\overline{\mathscr{D}}-\overline{\mathscr{D}}'\right\|_V.
\]
for every $\overline{\mathscr{D}},\overline{\mathscr{D}}'\in V$ and $\mathcal{V}\in\BC_{\RR}(\mathscr{X})$ with $\left\{c_{\mathscr{X}}(\nu)\,:\,\nu(\mathcal{V})>0\right\}\subset\Sigma$.
\end{proposition}

\begin{proof}
By extending $(V,\|\cdot\|_V)$ if necessary, we may assume that $V$ has a basis $\overline{\mathscr{A}}_1,\dots,\overline{\mathscr{A}}_r\in\aDiv(\mathscr{X};C^{\infty})$ such that $\overline{\mathscr{A}}_1,\dots,\overline{\mathscr{A}}_r$ are all effective.
We set
\[
\left\|a_1\overline{\mathscr{A}}_1+\dots+a_r\overline{\mathscr{A}}_r\right\|_1:=|a_1|+\dots+|a_r|
\]
for $a_1,\dots,a_r\in\RR$, and set
\[
\overline{\mathscr{D}}=\bm{a}\cdot\overline{\pmb{\mathscr{A}}},\quad\overline{\mathscr{D}}'=\bm{a}'\cdot\overline{\pmb{\mathscr{A}}},\quad\text{and}\quad\overline{\mathscr{A}}:=\overline{\mathscr{A}}_1+\dots+\overline{\mathscr{A}}_r.
\]
If $\bm{a}'=0$, then we can see $\avol\left(\overline{\mathscr{D}};\mathcal{V}\right)\leq C\|\bm{a}\|_1^{d+1}$ for
\begin{equation}
C:=\max\left\{1,\avol\left(\overline{\mathscr{A}}\right)\right\}
\end{equation}
by using Lemma~\ref{prop:positive:homogeneity}, so that we can assume that both $\bm{a}$ and $\bm{a}'$ are nonzero.

First, we assume $\bm{a},\bm{a}'\in\ZZ^r$ and $\bm{b}:=\bm{a}'-\bm{a}\geq 0$.
By Theorem~\ref{thm:Continuity_Est2}, we get a constant $C'\geq C$ depending only on $\overline{\pmb{\mathscr{A}}}$, $\Sigma$, and $\mathscr{X}$ such that
\begin{align*}
0 & \leq \als\left(m\OO_{\mathscr{X}}(\overline{\mathscr{D}}_2);m\mathcal{V}\right)-\als\left(m\OO_{\mathscr{X}}(\overline{\mathscr{D}}_1);m\mathcal{V}\right) \\
& \leq \als\left(\OO_{\mathscr{X}}(m\bm{a}\cdot\overline{\pmb{\mathscr{A}}}+m\max_i\left\{b_i\right\}\overline{\mathscr{A}});m\mathcal{V}\right)-\als\left(\OO_{\mathscr{X}}(m\bm{a}\cdot\overline{\pmb{\mathscr{A}}});m\mathcal{V}\right) \\
& \leq C'm^d(\|\bm{a}\|_1+\|\bm{b}\|_1)^d\left(m\|\bm{b}\|_1+\log(m\|\bm{a}\|_1)\right)
\end{align*}
for every $m\in\ZZ_{>0}$.
Hence
\begin{equation}
\label{eqn:continuity:est:4:1}
\avol\left(\overline{\mathscr{D}}_1;\mathcal{V}\right)\leq\avol\left(\overline{\mathscr{D}}_2;\mathcal{V}\right)\leq\avol\left(\overline{\mathscr{D}}_1;\mathcal{V}\right)+C'(\|\bm{a}\|_1+\|\bm{b}\|_1)^d\|\bm{b}\|_1.
\end{equation}
For general $\bm{a},\bm{a}'\in\ZZ^r$, we set $\bm{a}'':=\max\left\{\bm{a},\bm{a}'\right\}$ and $\overline{\mathscr{D}}'':=\bm{a}''\cdot\overline{\pmb{\mathscr{A}}}$.
By \eqref{eqn:continuity:est:4:1}
\begin{align*}
&\left|\avol\left(\overline{\mathscr{D}};\mathcal{V}\right)-\avol\left(\overline{\mathscr{D}}';\mathcal{V}\right)\right| \\
& \quad \leq\left|\avol\left(\overline{\mathscr{D}}'';\mathcal{V}\right)-\avol\left(\overline{\mathscr{D}};\mathcal{V}\right)\right|+\left|\avol\left(\overline{\mathscr{D}}'';\mathcal{V}\right)-\avol\left(\overline{\mathscr{D}}';\mathcal{V}\right)\right| \\
& \quad \leq C'(\|\bm{a}\|_1+\|\bm{a}''-\bm{a}\|_1)^d\|\bm{a}''-\bm{a}\|_1+C'(\|\bm{a}'\|_1+\|\bm{a}''-\bm{a}'\|_1)^d\|\bm{a}''-\bm{a}'\|_1 \\
& \quad \leq 2^dC'\max\left\{\|\bm{a}\|_1^d,\|\bm{a}'\|_1^d\right\}\|\bm{a}-\bm{a}'\|_1.
\end{align*}
Therefore, by using Lemma~\ref{prop:positive:homogeneity}, we can verify that the estimate is also true for every $\bm{a},\bm{a}'\in\QQ^r$.

Next, we show the estimate for every $\bm{a},\bm{a}'\in\RR^r$.

\begin{claim}
\label{clm:Continuity:Est:4:1}
Let $\left(\bm{p}^{(n)}\right)_{n\geq 1}$ be a sequence in $\QQ^r$ that converges to $\bm{a}\in\RR^r$.
Then
\[
\lim_{n\to+\infty}\avol\left(\bm{p}^{(n)}\cdot\overline{\pmb{\mathscr{A}}};\mathcal{V}\right)=\avol\left(\bm{a}\cdot\overline{\pmb{\mathscr{A}}};\mathcal{V}\right).
\]
\end{claim}

\begin{proof}[Proof of Claim~\ref{clm:Continuity:Est:4:1}]
Let $\left(\bm{b}^{(n)}\right)_{n\geq 1}$ and $\left(\bm{c}^{(n)}\right)_{n\geq 1}$ be two sequences in $\QQ^r$ such that
\[
b_i^{(1)}\leq b_i^{(2)}\leq\dots \leq b_i^{(n)}\leq \dots\leq a_i\leq \dots\leq c_i^{(n)}\leq \dots\leq c_i^{(2)}\leq c_i^{(1)}
\]
and $\lim_{n\to+\infty}\left|c_i^{(n)}-b_i^{(n)}\right|=0$ for $i=1,2,\dots,r$.
We then have
\begin{multline*}
\avol\left(\bm{b}^{(1)}\cdot\overline{\pmb{\mathscr{A}}};\mathcal{V}\right)\leq\avol\left(\bm{b}^{(2)}\cdot\overline{\pmb{\mathscr{A}}};\mathcal{V}\right)\leq\dots\leq \avol\left(\bm{b}^{(n)}\cdot\overline{\pmb{\mathscr{A}}};\mathcal{V}\right)\leq\dots \\
\leq\avol\left(\bm{a}\cdot\overline{\pmb{\mathscr{A}}};\mathcal{V}\right) \\
\leq\dots\leq\avol\left(\bm{c}^{(n)}\cdot\overline{\pmb{\mathscr{A}}};\mathcal{V}\right)\leq\dots\leq\avol\left(\bm{c}^{(2)}\cdot\overline{\pmb{\mathscr{A}}};\mathcal{V}\right)\leq\avol\left(\bm{c}^{(1)}\cdot\overline{\pmb{\mathscr{A}}};\mathcal{V}\right)
\end{multline*}
and
\[
\lim_{n\to+\infty}\left(\avol\left(\bm{c}^{(n)}\cdot\overline{\pmb{\mathscr{A}}};\mathcal{V}\right)-\avol\left(\bm{b}^{(n)}\cdot\overline{\pmb{\mathscr{A}}};\mathcal{V}\right)\right)=0
\]
by the above arguments.
Hence we have the required claim.
\end{proof}

We choose two sequences $\left(\bm{p}^{(n)}\right)_{n\geq 1}$ and $\left(\bm{q}^{(n)}\right)_{n\geq 1}$ in $\QQ^r$ such that $\lim_{n\to+\infty}\bm{p}^{(n)}=\bm{a}$ and $\lim_{n\to+\infty}\bm{q}^{(n)}=\bm{a}'$, respectively.
Then
\[
\left|\avol\left(\bm{p}^{(n)}\cdot\overline{\pmb{\mathscr{A}}};\mathcal{V}\right)-\avol\left(\bm{q}^{(n)}\cdot\overline{\pmb{\mathscr{A}}};\mathcal{V}\right)\right|\leq C\max\left\{\left\|\bm{p}^{(n)}\right\|_1^d,\left\|\bm{q}^{(n)}\right\|_1^d\right\}\left\|\bm{p}^{(n)}-\bm{q}^{(n)}\right\|_1
\]
for every $n\geq 1$ by the previous argument.
Taking $n\to+\infty$, we obtain the required result.
\end{proof}

\begin{proof}[Proof of Theorem~\ref{thm:Continuity_Est3}]
We may assume that $X$ is smooth.
In fact, let $\mu:\widetilde{X}\to X$ be a resolution of singularities of $X$, and regard $V$ as an $\RR$-subspace of $\aDiv_{\RR}^{\ell^1}(\widetilde{X})$ via $\aDiv_{\RR}^{\ell^1}(X)\to\aDiv_{\RR}^{\ell^1}(\widetilde{X})$.
Since $X$ is normal, we have $\avol\left(\mu^*\overline{D};\mathcal{V}\right)=\avol\left(\overline{D};\mathcal{V}\right)$ for every $\overline{D}\in V$ and $\mathcal{V}\in\BC_{\RR}(X)$.
Let $\overline{A}_1,\dots,\overline{A}_r\in\aDiv_{\RR}^{\ell^1}(X)$ be a basis for $V$, put
\[
\left\|a_1\overline{A}_1+\dots+a_r\overline{A}_r\right\|_1:=|a_1|+\dots+|a_r|
\]
for $a_1,\dots,a_r\in\RR$, and suppose that $\|\cdot\|_V$ is given as $\|\cdot\|_1$.
We can easily find a constant $B'\in\RR_{>0}$ such that $\vol(D)\leq B'$ for every $\overline{D}\in V$ with $\|\overline{D}\|_1\leq B$.

We put
\begin{equation}
\delta':=\frac{\varepsilon}{2(\dim X+1)[K:\QQ]B'(B+1)},
\end{equation}
and fix, for each $i$, $(\mathscr{X},\overline{\mathscr{A}}_i)\in\Mod_{\RR}(\overline{A}_i)$ such that $\overline{\mathscr{A}}_i\in\aDiv_{\RR}(\mathscr{X};C^{\infty})$ and such that $\left\|\overline{A}_i-\overline{\mathscr{A}}_i^{\rm ad}\right\|_{\ell^1}\leq\delta'$ by using the Stone--Weierstrass theorem and Proposition~\ref{prop:definition:of:ell1:adelic}.
Proposition~\ref{prop:difference:adding:cont:fcn:final:version} implies that
\begin{multline}
\left|\avol\left(\bm{a}\cdot\overline{\bm{A}}+(0,\bm{f});\mathcal{V}\right)-\avol\left(\bm{a}\cdot\overline{\pmb{\mathscr{A}}}^{\rm ad};\mathcal{V}\right)\right| \\
\leq\frac{1}{2}(\dim X+1)[K:\QQ]B'(\|\bm{a}\|_1+1)\delta\leq\frac{\varepsilon}{4}
\end{multline}
holds for every $\bm{a}\in\RR^r$ with $\|\bm{a}\|_1\leq B$, $\bm{f}\in\aC(X)$ with $\|\bm{f}\|_{\ell^1}\leq\delta'$, and $\mathcal{V}\in\BC_{\RR}(X)$.

Thanks to Proposition~\ref{prop:Continuity_Est4}, there is a constant $C_{\overline{\pmb{\mathscr{A}}},\Sigma}>0$ such that
\[
\left|\avol\left(\bm{a}\cdot\overline{\pmb{\mathscr{A}}}^{\rm ad};\mathcal{V}\right)-\avol\left(\bm{a}'\cdot\overline{\pmb{\mathscr{A}}}^{\rm ad};\mathcal{V}\right)\right|\leq C_{\overline{\pmb{\mathscr{A}}},\Sigma}\max\left\{\|\bm{a}\|_1^d,\|\bm{a}'\|_1^d\right\}\left\|\bm{a}-\bm{a}'\right\|_1
\]
for every $\bm{a},\bm{a}'\in\RR^r$ and $\mathcal{V}\in\BC_{\RR}(X)$ with $\left\{c_X(\nu)\,:\,\nu(\mathcal{V})>0\right\}\subset\Sigma$, so, if we set
\begin{equation}
\delta:=\min\left\{\delta',\frac{\varepsilon}{2C_{\pmb{\mathscr{A}},\Sigma}B^d}\right\},
\end{equation}
then
\begin{equation}
\left|\avol\left(\bm{a}\cdot\overline{\pmb{\mathscr{A}}}^{\rm ad};\mathcal{V}\right)-\avol\left(\bm{a}'\cdot\overline{\pmb{\mathscr{A}}}^{\rm ad};\mathcal{V}\right)\right|\leq\frac{\varepsilon}{2}
\end{equation}
for every $\bm{a},\bm{a}'\in\RR^r$ with $\max\{\|\bm{a}\|_1,\|\bm{a}'\|_1\}\leq B$ and $\|\bm{a}-\bm{a}'\|_1\leq\delta$.
All in all, we have
\begin{align*}
& \left|\avol\left(\bm{a}\cdot\overline{\bm{A}}+(0,\bm{f});\mathcal{V}\right)-\avol\left(\bm{a}'\cdot\overline{\bm{A}};\mathcal{V}\right)\right| \\
& \quad\leq\left|\avol\left(\bm{a}\cdot\overline{\bm{A}}+(0,\bm{f});\mathcal{V}\right)-\avol\left(\bm{a}\cdot\overline{\pmb{\mathscr{A}}}^{\rm ad};\mathcal{V}\right)\right| \\
& \qquad+\left|\avol\left(\bm{a}\cdot\overline{\pmb{\mathscr{A}}}^{\rm ad};\mathcal{V}\right)-\avol\left(\bm{a}'\cdot\overline{\pmb{\mathscr{A}}}^{\rm ad};\mathcal{V}\right)\right| +\left|\avol\left(\bm{a}'\cdot\overline{\pmb{\mathscr{A}}}^{\rm ad};\mathcal{V}\right)-\avol\left(\bm{a}'\cdot\overline{\bm{A}};\mathcal{V}\right)\right| \\
& \quad\leq\varepsilon
\end{align*}
as required.
\end{proof}

Theorem~\ref{thm:Continuity_Est3} implies the following corollaries.

\begin{corollary}
For a $(\overline{D};\mathcal{V})\in\aBDiv_{\RR,\RR}^{\ell^1}(X)$, the following are equivalent.
\begin{enumerate}
\item $\avol\left(\overline{D};\mathcal{V}\right)>0$.
\item For any $\overline{A}\in\aDiv_{\RR}^{\ell^1}(X)$ with $\avol\left(\overline{A}\right)>0$, there exists a $t\in\RR_{>0}$ such that $(\overline{D}-t\overline{A};\mathcal{V})\geq 0$.
\end{enumerate}
\end{corollary}

\begin{corollary}
\label{cor:positive:homogeneity:general}
For any $(\overline{D};\mathcal{V})\in\aBDiv_{\RR,\RR}^{\ell^1}(X)$ and $p\in\RR_{>0}$, one has
\[
\avol\left(p\overline{D};p\mathcal{V}\right)=p^{\dim X+1}\avol\left(\overline{D};\mathcal{V}\right).
\]
\end{corollary}

\begin{proof}
We may assume that $X$ is smooth.
Let $V$ be a finite-dimensional $\RR$-subspace of $\aBDiv_{\RR,\RR}^{\ell^1}(X)$ such that $V$ has a basis $\overline{A}_1,\dots,\overline{A}_r\in\aDiv_{\QQ}^{\ell^1}(X)$ and such that $\overline{D}=\bm{a}\cdot\overline{\bm{A}}\in V$ for an $\bm{a}\in\RR^r$.
Let $\left(\bm{b}^{(n)}\right)_{n\geq 1}$ be a sequence in $\QQ^r$ that converges to $\bm{a}$.
By the Stone--Weierstrass theorem and Proposition~\ref{prop:definition:of:ell1:adelic}, one finds, for each $i$, a sequence $\left((\mathscr{X}_n,\overline{\mathscr{A}}_{in})\right)_{n\geq 1}$ in $\Mod_{\QQ}(\overline{A}_i)$ such that $\overline{\mathscr{A}}_{in}\in\aDiv_{\QQ}(\mathscr{X}_n;C^{\infty})$, such that $\overline{\mathscr{A}}_{i1}^{\rm ad}\leq\overline{\mathscr{A}}_{i2}^{\rm ad}\leq\dots$, and such that $\left\|\overline{A}_i-\overline{\mathscr{A}}_{in}^{\rm ad}\right\|_{\ell^1}\to 0$ as $n\to+\infty$.
By Lemma~\ref{prop:positive:homogeneity},
\[
\avol\left(p\bm{b}^{(n)}\cdot\overline{\pmb{\mathscr{A}}}_n^{\rm ad};p\mathcal{V}\right)=p^{\dim X+1}\avol\left(\bm{b}^{(n)}\cdot\overline{\pmb{\mathscr{A}}}_n^{\rm ad};\mathcal{V}\right)
\]
for $p\in\QQ_{>0}$ and $n\geq 1$.
Taking $n\to+\infty$ (Theorem~\ref{thm:Continuity_Est3}), we obtain the equality for every $p\in\QQ_{>0}$.

To show the corollary, we note that the inequality $\leq$ is obvious.
We choose an decreasing sequence $(q_n)_{n\geq 1}$ in $\QQ_{>0}$ that converges to $p$.
Then
\[
\avol\left(q_n\overline{D};p\mathcal{V}\right)\geq\avol\left(q_n\overline{D};q_n\mathcal{V}\right)=q_n^{\dim X+1}\avol\left(\overline{D};\mathcal{V}\right)
\]
for $n\geq 1$.
By taking $n\to+\infty$, we conclude the proof by Theorem~\ref{thm:Continuity_Est3}.
\end{proof}

\begin{corollary}
\label{cor:invariant:under:rational:fcn}
For any $(\overline{D};\mathcal{V})\in\aBDiv_{\RR,\RR}^{\ell^1}(X)$ and $\phi\in\Rat(X)^{\times}\otimes_{\ZZ}\RR$, one has
\[
\avol\left(\overline{D}+\widehat{(\phi)};\mathcal{V}\right)=\avol\left(\overline{D};\mathcal{V}\right).
\]
\end{corollary}

\begin{proof}
We write $\phi=\phi_1^{a_1}\cdots\phi_r^{a_r}$ with $a_i\in\RR$ and $\phi_i\in\Rat(X)$.
Let $V$ be the $\RR$-subspace of $\aDiv_{\RR}^{\ell^1}(X)$ generated by $\phi_1,\dots,\phi_r$.
For each $i$, we choose a sequence $\left(b_i^{(n)}\right)_{n\geq 1}$ in $\QQ$ such that $b_i^{(n)}\to a_i$ as $n\to+\infty$.
By homogeneity (Corollary~\ref{cor:positive:homogeneity:general}), we have
\[
\avol\left(\overline{D}+\sum_{i=1}^rb_i^{(n)}\widehat{(\phi_i)};\mathcal{V}\right)=\avol\left(\overline{D};\mathcal{V}\right)
\]
for every $n\geq 1$.
Taking $n\to+\infty$, we obtain the required assertion by Theorem~\ref{thm:Continuity_Est3}.
\end{proof}

\begin{corollary}
For each $\mathcal{V}\in\BC_{\RR}(X)$, the arithmetic volume function induces a continuous function $\aCl_{\RR}^{\ell^1}(X)\to\RR_{\geq 0}$, $\overline{D}\mapsto\avol\left(\overline{D};\mathcal{V}\right)$.
\end{corollary}

\begin{proof}
By using Corollary~\ref{cor:invariant:under:rational:fcn}, we can obtain the required map.
To show the continuity, let $q:\aDiv_{\RR}^{\ell^1}(X)\to\aCl_{\RR}^{\ell^1}(X)$ be the natural projection and fix a section $\iota':\Cl_{\RR}(X)\to\aDiv_{\RR}^{\ell^1}(X)$ of $\zeta$.
Let $V$ be the image of $\iota'$ and let $\|\cdot\|$ be a norm on $\Cl_{\RR}(X)$.
Set
\[
\left\|\overline{D}\right\|_{\iota',\|\cdot\|}:=\left\|\zeta(\overline{D})\right\|+\left\|\overline{D}-\iota'\circ\zeta(\overline{D})\right\|_{\ell^1}
\]
for $\overline{D}\in V\oplus\aC(X)$, and set $\iota:=q\circ\iota'$.
We then have $\left\|\overline{D}\right\|_{\iota',\|\cdot\|}=\left\|q(\overline{D})\right\|_{\iota,\|\cdot\|}$ for every $\overline{D}\in V$.
Hence the assertion results from Theorem~\ref{thm:Continuity_Est3}.
\end{proof}

\section*{Acknowledgement}
This work was supported by JSPS KAKENHI Grant Number 16K17559.
The author is grateful to Professors Namikawa, Yoshikawa, and Moriwaki and Kyoto University for the financial supports.

\bibliography{ikoma}
\bibliographystyle{plain}

\end{document}